\documentclass[12pt]{article} 

\usepackage{color}
\usepackage[margin=1in]{geometry} 
\usepackage{graphicx} 
\usepackage{amsmath} 
\usepackage{amsfonts} 
\usepackage{amsthm} 
\usepackage{mathrsfs}
\usepackage{mathtools}
\usepackage{enumitem}
\usepackage{hyperref}
\newtheorem{theorem}{Theorem}[section]
\newtheorem{lemma}[theorem]{Lemma}

\newtheorem{corollary}[theorem]{Corollary}

\newtheorem{definition}[theorem]{Definition}
\newtheorem{example}[theorem]{Example}
\newtheorem{remark}[theorem]{Remark}
\newcommand{\de}{\mathrm{d}}
\DeclareMathAlphabet{\mathpzc}{OT1}{pzc}{m}{it}

\providecommand{\keywords}[1]{\textbf{\textit{Keywords: }}#1}

\title{Weak Solutions to Vlasov-McKean Equations under Lyapunov-Type Conditions}

\author{Sima Mehri\footnotemark[1]\ \footnotemark[2]\ \footnotemark[3] \and Wilhelm Stannat\footnotemark[1]}

\begin{document}

\maketitle
\renewcommand{\thefootnote}{\fnsymbol{footnote}}
\footnotetext[1]{Institut f\"ur Mathematik, Technische Universit\"at Berlin, D-10623 Berlin, Germany}
\footnotetext[2]{Department of Mathematical Sciences, Sharif University of Technology, Tehran, Iran}
\footnotetext[3]{The work of this author was supported by the Hilda Geiringer Scholarship awarded by the Berlin Mathematical School}
\renewcommand{\thefootnote}{\arabic{footnote}}

\begin{abstract}
	We present a Lyapunov type approach to the problem of existence and uniqueness of general 
	law-dependent stochastic differential equations. In the existing literature most results concerning 
	existence and uniqueness are obtained under regularity assumptions of the coefficients w.r.t the 
	Wasserstein distance. Some existence and uniqueness results for irregular coefficients have been 
	obtained by considering the total variation distance. Here we extend this approach to the control of 
	the solution in some weighted total variation distance, that allows us now to derive a 
	rather general weak uniqueness result, merely assuming measurability and certain integrability on 
	the drift coefficient and some non-degeneracy on the dispersion coefficient. We also present an 
	abstract weak existence result for the solution of law-dependent stochastic differential equations with 
	merely measurable coefficients, based on an approximation with law-dependent stochastic differential 
	equations with regular coefficients under Lyapunov type assumptions.
\end{abstract}

\section{Introduction}

The purpose of this paper is to provide general existence and uniqueness results for the solution of 
Vlasov-McKean equations, and more general law-dependent stochastic differential equations, using a 
Lyapunov approach. The existence and uniqueness of solutions of Vlasov-McKean equations under global 
Lipschitz conditions is well-known. Surprisingly, uniqueness fails under local Lipschitz assumptions 
(see \cite{scheutzow1987uniqueness}). However, in these counterexamples, the noise is degenerate 
(in fact zero). As the following Example of uniqueness with merely measurable coefficients shows, the situation changes, if the noise becomes non-degenerate. 

\begin{example}
	On the complete probability space $(\Omega,\mathcal{F},(\mathcal{F}_t)_{t\geq 0},\mathbb{P})$ with real valued $\left(\mathcal{F}_t\right)_{t\geq 0}$-Wiener process $\left(W_t\right)_{t\geq 0}$ on $\mathbb{R}$, consider the following Vlasov-McKean equation 
	\begin{equation}\label{example-equ}
	\begin{dcases} 
	\de X_t=\mathbb{E}\left(h(X_t)\right)\de t+\de W_t 
	\\ X_0=\xi   
	\end{dcases}
	\end{equation}
	with measurable $h$ satisfying the growth condition $\left\lvert h(x)\right\rvert 
	\leq Ce^{\frac{x^2}{2T}}$ for some $T > 0$. Let $\mu_0 :=\mathbb{P}\circ \xi^{-1}$ be absolutely continuous with continuous differentiable density, and define
	\begin{align*}
	\phi_h(t,x) 
	& := \int_{\mathbb{R}}\int_{\mathbb{R}}\frac{1}{\sqrt{2\pi t}}h(x_0 + x + w)e^{-\frac{w^2}{2t}}\de w\mu_0(\de x_0)  \\ 
	& = \int_{\mathbb{R}}\int_{\mathbb{R}}\frac{1}{\sqrt{2\pi t}}h(x_0 + w)e^{-\frac{(w-x )^2}{2t}}\de w\mu_0(\de x_0)\, . 
	\end{align*}
	Then for $t<T$, $x\mapsto\phi_h(t,x)$ is continuous differentiable, hence locally Lipschitz continuous. 
	Let $X_t=\xi+g(t)+W_t$ be a solution of \eqref{example-equ}, then
	\begin{align*}
	g^\prime(t)\de t + \de W_t 
	& = \de X_t=\mathbb{E}\left( h(X_t)\right)\de t + \de W_t \\ 
	& = \mathbb{E}\left(h(\xi+g(t)+W_t)\right)\de t+\de W_t \\ 
	& = \int_{\mathbb{R}}\int_{\mathbb{R}}\frac{1}{\sqrt{2\pi t}}h(x_0+g(t)+w)e^{-\frac{w^2}{2t}}\de  
	w\mu_0(\de x_0)\de t+\de W_t \\ 
	& = \phi_h(t,g(t)) \de t + \de W_t.
	\end{align*}
	So $g:[0,T)\to \mathbb{R}$ is the unique solution to the equation $g^\prime(t)=\phi_h(t,g(t))$, with  
	initial value $g(0)=0$. Therefore equation \eqref{example-equ} has a unique strong solution on $[0,T)$. 
\end{example} 

Hence there is a considerable interest in relaxing the assumptions on the coefficients of 
Vlasov-McKean equations. Strong well-posedness of Vlasov-McKean equation with H\"older drift and Lipschitz 
dispersion coefficient has been obtained in \cite{de2019strong}. Strong existence and uniqueness of solutions 
to the Vlasov-McKean equation under one-sided Lipschitz continuity for the drift and Lipschitz 
continuous dispersion coefficient have been obtained in \cite{dos2019freidlin}. 
The paper \cite{wang2018distribution} considers strong well-posedness of distribution dependent 
stochastic differential equations with one-sided Lipschitz continuous drift and Lipschitz-continuous 
dispersion coefficients, \cite{huang2019nonlinear} generalizes the latter result to path-distribution dependent stochastic differential equations. 

Weak existence and strong uniqueness of solutions to the Vlasov-McKean equation with continuous 
coefficients have been obtained with the help of a Lyapunov method in \cite{hammersley2018mckean}. 
The recent preprint \cite{mishura2016existence} proves weak and strong well-posedness of the solutions of 
Vlasov-McKean equations under non-degeneracy assumptions on the noise term with even non-regular drift of at 
most linear growth. 

Existence and uniqueness of weak solutions of Vlasov-McKean equations have been obtained in 
\cite{lacker2018strong}, with regularity of the coefficients w.r.t. the total variation distance. 
\cite{bauer2018strong} obtains existence and uniqueness of weak and strong solutions of Vlasov-McKean 
equations with additive noise and drift coefficients that can be decomposed into bounded measurable 
part and a part that is Lipschitz continuous w.r.t. the Kantorovich distance. 

The paper \cite{huang2018distribution} contains an existence result of a weak solution of a 
distribution-dependent stochastic differential equation with merely measurable coefficients 
based on an approximation with stochastic differential equations with Lipschitz continuous 
coefficients. This result requires uniform boundedness of the diffusion term. 

In the present paper now, we will extend the result for the existence of weak solutions to 
Vlasov-McKean equations with measurable coefficients and uniformly non-degenerate and merely integrable 
diffusion matrix (see the Theorem \ref{abstractexistence}). The abstract conditions in this theorem 
can be verified with the help of a Lyapunov type growth condition on the coefficients in Theorem 
\ref{theorem4-1}. Sufficient conditions, in terms of the coefficients only, are presented in Corollary 
\ref{theorem4}. 

We also obtain a corresponding uniqueness result for weak solutions of functional law-dependent 
stochastic differential equations under Lyapunov type growth conditions on the coefficients (see 
Corollary \ref{theorem2-1}), based on an abstract stability result for weak solutions w.r.t. a weighted total 
variation distance (see Theorem \ref{prop1}). Two sets of sufficient conditions in terms of 
the coefficients are presented in Example \ref{ex2.6}. Our uniqueness results generalize the 
corresponding result obtained in \cite{mishura2016existence} not only w.r.t. the general law-dependence 
but also w.r.t. the more general growth conditions. In \cite{mishura2016existence}, only linear growth 
is allowed. Stability results for Vlasov-McKean equations w.r.t. weighted total variation distances 
have been obtained previously in the references \cite{bogachev2016distances,manita2015uniqueness}, 
using an analytic approach, that cannot, however, cover general functional law-dependent
stochastic differential equations considered in the present work.

\section{Uniqueness Result}
Let $\mathcal{M}$ be the space of signed measures on $\left(\mathbb{R}^d,\mathcal{B}(\mathbb{R}^d)\right)$. Given a measurable function $\phi:\mathbb{R}^d\to(0,\infty)$, we define the $\phi$-weighted total variation of $\mu\in\mathcal{M}$ by
\[\left\lVert \mu  \right\rVert_\phi:=\int_{\mathbb{R}^d}\phi(y)\left\lvert \mu \right\rvert(\de y)\] 
Here $\left\lvert \mu\right\rvert$ denotes the total variation measure associated with $\mu$. For a continuous function $\phi$, this norm is lower semi-continuous with respect to the weak topology by the following Lemma.

\begin{lemma} 
	\label{lemma}
	Let $\phi:\mathbb{R}^d\to(0,\infty)$ be continuous and assume that the sequence of signed measures $\left(\mu_n\right)_{n\in\mathbb{N}}$ in $\mathcal{M}$ converges weakly to the measure $\mu$ and assume that $\phi\in \mathcal{L}^1 (|\mu |)$. Then
	\[
	\left\lVert \mu  \right\rVert_\phi\leq \liminf_{n\to\infty}\left\lVert \mu_n \right\rVert_\phi. 
	\]
\end{lemma}

\begin{proof}
	Using the Hahn decomposition theorem we can find a measurable subset $A\in \mathcal{B}(\mathbb{R}^d)$ 
	such that $|\mu| = \mu_A - \mu_{A^c}$, where $\mu_A (B) = \mu (B\cap A)$ and  
	$\mu_{A^c} (B) = \mu (B\cap A^c)$, $B\in\mathcal{B} (\mathbb{R}^d)$, are finite nonnegative Borel 
	measures. Let $\varepsilon > 0$ be arbitrary. Since $\phi\in \mathcal{L}^1 (|\mu |)$ we can find 
	$R > 0$ such that 
	\[
	\left\lVert \mu  \right\rVert_\phi 
	\leq \int_{\mathbb{R}^d}\left(\phi (y)\wedge R\right) 
	\left(\mathbf{1}_A(y)-\mathbf{1}_{A^c}(y)\right)\mu (\de y)+\varepsilon . 
	\]
	Since $(\phi\wedge R)\, \de\mu$ is a finite Borel measure we can find a continuous function  
	$\psi : \mathbb{R}^d \to [-1,1]$ satisfying 
	$$ 
	\int_{\mathbb{R}^d}\left(\phi (y)\wedge R\right) 
	\left(\mathbf{1}_A(y)-\mathbf{1}_{A^c}(y)\right)\mu (\de y) 
	\le \int_{\mathbb{R}^d}\left(\phi (y)\wedge R\right) \psi (y)\mu (\de y) +\varepsilon \, . 
	$$ 
	Consequently,  
	\begin{align*}
	\left\lVert \mu  \right\rVert_\phi 
	& \leq \int_{\mathbb{R}^d}\left(\phi(y)\wedge R\right)\psi(y)\mu (\de y)+2\varepsilon 
	\leq \lim_{n\to\infty}\int_{\mathbb{R}^d}\left(\phi(y)\wedge R\right)\psi (y) 
	\mu_n (\de y) + 2\varepsilon \\ 
	& \leq\liminf_{n\to\infty}\left\lVert \mu_n \right\rVert_\phi + 2\varepsilon.
	\end{align*}
	Since $\varepsilon > 0$ is arbitrary, this implies the assertion. 
\end{proof}

Fix $T,\tau>0$. Let $\mathfrak{M}$ be the Borel $\sigma$-algebra induced by the weak topology on $\mathcal{M}$. Let us define
\[\mathcal{M}_T:=\left\lbrace \mu: [-\tau,T]\to\mathcal{M}; \mu \text{ is } \mathcal{B}\left([-\tau,T]\right)/\mathfrak{M}\text{-measurable} \right\rbrace.\] 

Let $(W_t)_{t\geq 0}$ be the standard Brownian motion on $\mathbb{R}^{d_1}$. We consider the nonlinear equation
\begin{equation}\label{equ1}
\begin{dcases}\de X_t=b(t,X,\mu)\de t+\sigma(t,X)\de W_t,\quad t\in [0,T],\\X_t=\xi_t, \quad t\in[-\tau,0],
\\\mu\in \mathcal{M}_T, \mu_s=\mathcal{L}(X_s), \text{ where }\mathcal{L}(X_s) \text{ denotes the law of }X_s, s\in[-\tau,T]\end{dcases}
\end{equation}
with initial condition $\xi\in C([-\tau,0];\mathbb{R}^d)$, independent of $(W_t)_{t\geq 0}$, where $b\equiv \sigma \tilde{b}$ and
\[\begin{dcases}\tilde{b}:[0,T]\times C([-\tau,T],\mathbb{R}^d)\times \mathcal{M}_T\to \mathbb{R}^{d_1},\\\sigma:[0,T]\times C([-\tau,T],\mathbb{R}^d)\to \mathbb{R}^{d\times d_1}\end{dcases}\]
are measurable functions and adapted, i.e. $\tilde{b}(t,x,\mu)$ and $\sigma(t,x)$ depend only on the path of $x$ and $\mu$ on $[-\tau,t]$. This equation is called a Vlasov-McKean equation. 
\begin{definition}\label{definition}
	We say that equation \eqref{equ1} has a weak solution on $[0,T]$ with initial distribution $\Xi$ on $C([-\tau,0],\mathbb{R}^d)$ if there exist a probability space $\left(\Omega,\mathcal{F},\left(\mathcal{F}_t\right)_{t\geq 0},\mathbb{P}\right)$, an $\left(\mathcal{F}_t\right)_{t\geq 0}$-Wiener process $(W_t)_{t\geq 0}$ on $\mathbb{R}^{d_1}$, an $\mathcal{F}_0$-measurable random variable $\xi \in C([-\tau,0],\mathbb{R}^d)$ with the law $\Xi$, and an $\left(\mathcal{F}_t\right)_{t\geq 0}$-adapted stochastic process $X\in C([-\tau,T],\mathbb{R}^d)$ such that 
	\begin{equation}
	\begin{dcases}X_t=\xi(0)+ \int_0^t b(s,X,\mu)\de s+\int_0^t\sigma(s,X)\de W_s,\quad t\in [0,T],\\X_t=\xi_t, \quad t\in[-\tau,0],
	\\\mu\in \mathcal{M}_T, \mu_s=\mathcal{L}(X_s),\end{dcases}
	\end{equation}
	which requires that the integrals are well defined, i.e.,
	\begin{equation}\label{integsigmab}
	\int_0^T\left\lvert b(s,X,\mu)\right\rvert+ \left\lvert\sigma(s,X)\right\rvert^2\de s<\infty, \quad \mathbb{P}\text{-a.s.}
	\end{equation}
\end{definition}
\begin{remark}
	Note that by Levy's theorem on characterization of Brownian motion, for any $\left(\mathcal{F}_t\right)_{t\geq 0}$-Wiener process $(W_t)_{t\geq 0}$, $W_t-W_s$ is independent of $\mathcal{F}_s$. Specially $(W_t)_{t\geq 0}$ is independent of $\mathcal{F}_0$, that means in Definition \ref{definition}, $\xi$ is in fact independent of $(W_t)_{t\geq 0}$.
\end{remark}
We will first state an abstract uniqueness result for the weak solution to the Vlasov-McKean equation 
\eqref{equ1} in the following theorem, that is based on an estimate of the distance of the laws of two 
weak solutions with different drift and same dispersion coefficient w.r.t. the weighted total variation 
distance introduced above. 

\begin{theorem} 
	\label{prop1}
	Suppose that equation
	\begin{equation}\label{equ0}
	\begin{dcases}\de X^0_t=\sigma(t,X^0)\de W_t,\quad t\in[0,T],\\X^0_t=\xi_t, \quad t\in[-\tau,0],\end{dcases}
	\end{equation}
	has a unique strong solution on the probability space $\left(\Omega,\mathcal{F},\left(\mathcal{F}_t\right)_{t\geq 0},\mathbb{P}\right)$ for some $\mathcal{F}_0$-measurable random variable $\xi\in C([-\tau,0],\mathbb{R}^d)$. Let $\tilde{b}_1,\tilde{b}_2:[0,T]\times C([-\tau,T],\mathbb{R}^d)\to \mathbb{R}^{d_1}$ be such that for $i=1,2$,
	\begin{equation}\label{integrability-b}
	\int_0^T\left\lvert \tilde{b}_i(s,X^0)\right\rvert^2\de s<\infty, \quad \mathbb{P}{\text-}a.s.
	\end{equation}
	and $\tilde{b}_i(t,x)$ depends only to the path of $x$ on $[-\tau,t]$. Let  $X^{(i)}$, $i=1,2$, defined on the probability spaces $\left(\Omega^{(i)},\mathcal{F}^{(i)},\left(\mathcal{F}_t^{(i)}\right)_{t\geq 0},\mathbb{Q}^{(i)}\right)$  be weak solutions to the equations
	\begin{equation}\label{equ2}
	\begin{dcases}\de X^{(i)}_t=b_i(t,X^{(i)})\de t+\sigma(t,X^{(i)})\de W^{(i)}_t,\quad t\in[0,T]\\X^{(i)}_t=\xi^{(i)}_t, \quad t\in[-\tau,0],\end{dcases} 
	\end{equation}
	where $b_i\equiv \sigma \tilde{b}_i$, and $\xi^{(i)}$ is independent of $W^{(i)}$ and has the same law as $\xi$. Assume that for $i=1,2$, $X^{(i)}$ satisfies for $j=1,2$,
	\begin{equation}\label{finite}
	\int_0^T\left\lvert \tilde{b}_j(s,X^{(i)})\right\rvert^2\de s<\infty,\quad \mathbb{Q}^{(i)}\text{-}a.s.
	\end{equation}
	If $\mu^{(i)}_t$, $i=1,2$ denotes the law of $X^{(i)}_t$, then for any continuous function $\phi:\mathbb{R}^d\to (0,\infty)$
	\begin{align}\label{result1}
	\left\lVert \mu^{(1)}_t-\mu^{(2)}_t\right\rVert_{\phi}\nonumber&\leq\sum_{i=1}^2\mathbb{E}_{\mathbb{Q}^{(i)}}\left[\phi\left(X^{(i)}_{t}\right)\int_0^{t}\left\lvert \tilde{b}_1\left(s, X^{(i)}\right)-\tilde{b}_2\left(s,X^{(i)}\right)\right\rvert^2\de s\right]\nonumber\\&\quad+\sum_{i=1}^2\left(\mathbb{E}_{\mathbb{Q}^{(i)}}\left[\phi^2\left(X^{(i)}_{t}\right)\right]\right)^{1/2}\left(\mathbb{E}_{\mathbb{Q}^{(i)}}\left[\int_0^{t}\left\lvert \tilde{b}_1\left(s, X^{(i)}\right)-\tilde{b}_2\left(s,X^{(i)}\right)\right\rvert^2\de s\right]\right)^{1/2}.
	\end{align}
	In addition, let $b_i(t,x):=b\left(t,x,\mu^{(i)}\right)$ and assume that there exist measurable functions $\varphi:[0,T]\to C(\mathbb{R}^d,(0,\infty))$ and $\psi:[0,T]\times C([-\tau,T],\mathbb{R}^d)\to [0,\infty)$ and  an increasing positive valued function $g$ with $\int_{0^+} \frac{1}{g(u)}\de u =\infty$ such that for every $\mu,\nu \in\mathcal{M}_T$ with $\mu\vert_{[-\tau,0]}=\nu\vert_{[-\tau,0]}$,
	\begin{equation}\label{ineq-b}
	\left\lvert \tilde{b}\left(t,x,\mu\right)-\tilde{b}\left(t,x,\nu\right)\right\rvert\leq \psi(t,x)g^{1/2}\left(\sup_{s\in[0,t]}\left\lVert \mu_s-\nu_s\right\rVert_{\varphi_s}^2\right),
	\end{equation}
	Then $\mathbb{Q}^{(1)}\circ \left(X^{(1)}\right)^{-1}= \mathbb{Q}^{(2)}\circ \left(X^{(2)}\right)^{-1}$ provided that
	\begin{align}\label{finiteint}
	& \int_0^T\sup_{t\in[s,T]}\Bigg\lbrace\sum_{i=1}^2\mathbb{E}_{\mathbb{Q}^{(i)}} 
	\left[\varphi_t\left(X^{(i)}_{t}\right)\int_0^{t}\left\lvert \tilde{b}\left(u, X^{(i)},
	\mu^{(1)}\right)-\tilde{b}\left(u,X^{(i)},\mu^{(2)}\right)\right\rvert^2\de u\right] \cdot\nonumber \\ 
	& \quad\cdot\mathbb{E}_{\mathbb{Q}^{(i)}}\left[\varphi_t\left(X^{(i)}_{t}\right) \psi^2 
	(s,X^{(i)})\right] + \sum_{i=1}^2\mathbb{E}_{\mathbb{Q}^{(i)}} 
	\left[\varphi_t^2\left(X^{(i)}_{t} \right) \right]\cdot 
	\mathbb{E}_{\mathbb{Q}^{(i)}} \left[\psi^2(s,X^{(i)})\right]\Bigg\rbrace\de s < \infty
	\end{align} 
	for $i=1,2$.
\end{theorem}

\begin{proof}
	Let $X^0$ be the unique strong solution to the following equation
	\begin{equation*}
	\begin{dcases}\de X^0_t=\sigma(t,X^0)\de W_t,\quad t\in[0,T],\\X^0_t=\xi_t, \quad t\in[-\tau,0].\end{dcases}
	\end{equation*}
	Using Girsanov transformation, it turns out that equation \eqref{equ2} has at most one weak solution satisfying \eqref{finite}. Let us define the stopping time $\tau_n$ as
	\[ 
	\tau_n := \inf\left\lbrace t \geq 0,\min_{i=1,2}\int_0^t\left\lvert\tilde{b}_i(s,X^0)\right\rvert^2 
	\de s>n\right\rbrace. 
	\]
	Then the following process for $i=1,2$ is a martingale
	\[ 
	M^{(i)}_{t\wedge\tau_n}:=\exp\left(\int_0^{t\wedge \tau_n} \tilde{b}_i(s,X^0)\cdot\de W_s 
	-\frac{1}{2}
	\int_0^{t\wedge \tau_n}\left\lvert \tilde{b}_i(s,X^0)\right\rvert^2\de s\right) 
	,\quad  t\in[0,T].\]
	Let $\mathbb{P}^{i,n}$ be the probability measure  with density
	\[ 
	\frac{\de \mathbb{P}^{i,n}}{\de \mathbb{P}}\bigg\vert_{\mathcal{F}_{T}}=M^{(i)}_{T\wedge\tau_n}. 
	\]
	By Girsanov theorem, the process
	\[\tilde{W}^{(i)}_{t\wedge \tau_n}=W_{t\wedge \tau_n}-\int_0^{t\wedge \tau_n}\tilde{b}_i(s,X^0)\de s, \quad t\in[0,T],\]
	with respect to the probability measure $\mathbb{P}^{i,n}$ for $i=1,2$, is a standard Brownian motion 
	on $\mathbb{R}^{d_1}$ until time $\tau_n$ and we have
	\[X^0_{t\wedge \tau_n}=\xi_0+\int_0^{t\wedge \tau_n} b_i(s,X^0)\de s+\int_0^{t\wedge \tau_n} 
	\sigma(s,X^0)\de \tilde{W}^{(i)}_s.\]
	Let
	\[\zeta^{(i)}_n:=\inf\left\lbrace t \geq 0,\min_{j=1,2} \int_0^t 
	\left\lvert \tilde{b}_j(s,X^{(i)})\right\rvert^2\de s>n\right\rbrace.\]
	Then if we define 
	\[\frac{\de \mathbb{Q}^{i,n}}{\de \mathbb{Q}^{(i)}} 
	\bigg\vert_{\mathcal{F}^{(i)}_{T\wedge \zeta^{(i)}_n}}:=\exp\left(-\int_0^{T\wedge \zeta^{(i)}_n} 
	\tilde{b}_i(s,X^{(i)})\cdot\de W^{(i)}_s-\frac{1}{2}\int_0^{T\wedge \zeta^{(i)}_n} 
	\left\lvert \tilde{b}_i(s,X^{(i)})\right\rvert^2\de s\right)\]
	then 
	\[ 
	\bar{W}^{(i)}_{t\wedge \zeta^{(i)}_n}=W^{(i)}_{t\wedge \zeta^{(i)}_n}+\int_0^{t\wedge 
		\zeta^{(i)}_n}\tilde{b}_i(s,X^{(i)})\de s,\quad  t\in[0,T], 
	\]
	with respect to the probability measure $\mathbb{Q}^{i,n}$, for $i=1,2$, is a standard Brownian motion 
	in $\mathbb{R}^{d_1}$  until time $\zeta^{(i)}_n$ and we have that
	\[X^{(i)}_{t\wedge\zeta^{(i)}_n}=\xi^{(i)}_0+\int_0^{t\wedge\zeta^{(i)}_n}\sigma(s,X^{(i)})\de 
	\bar{W}^{(i)}_s,\]
	and $\left(X^{(i)}_t\right)_{-\tau\leq t\leq 0}$  w.r.t $\mathbb{Q}^{i,n}$ has the same law as $\xi$.
	Since equation \eqref{equ0} has a unique strong solution, there exists a measurable function
	\[
	F:C\left([-\tau,0],\mathbb{R}^d\right)\times C\left([0,T],\mathbb{R}^{d_1}\right)\to C\left([-\tau,T],
	\mathbb{R}^d\right) 
	\] 
	such that $X^0=F\left(\xi, W\right)$ and similarly $X^{(i)}_{\cdot\wedge \zeta^{(i)}_n} 
	= F\left(\xi^{(i)},\bar{W}^{(i)}_{\cdot\wedge \zeta^{(i)}_n}\right)$. Hence for $-\tau\leq t_0 
	\leq t_1\leq \cdots \leq t_m\leq T$,
	\begin{align*}
	\mathbb{Q}^{(i)} & \left[\left(X^{(i)}_{t_0\wedge \zeta^{(i)}_n}, \ldots,  
	X^{(i)}_{t_m\wedge \zeta^{(i)}_n}\right)\in \Gamma\right]
	\\ 
	& \qquad =\int_{\Omega^{(i)}} \exp\left(\int_0^{T\wedge \zeta^{(i)}_n} \tilde{b}_i(s,X^{(i)})\cdot\de 
	\bar{W}^{(i)}_s-\frac{1}{2}\int_0^{T\wedge \zeta^{(i)}_n}\left\lvert \tilde{b}_i (s,X^{(i)}) 
	\right\rvert^2\de s\right) \cdot \\ 
	& \qquad\qquad\qquad \cdot\mathbf{1}_{\left\lbrace\left(X_{t_0\wedge \zeta^{(i)}_n}^{(i)},\ldots, 
		X^{(i)}_{t_m\wedge \zeta^{(i)}_n}\right)\in \Gamma\right\rbrace}    
	\de \mathbb{Q}^{i,n} \\ 
	& \qquad =\int_{\Omega} \exp\left(\int_0^{T\wedge \tau_n} \tilde{b}_i(s,X^0)\cdot\de W_s 
	- \frac{1}{2}\int_0^{T\wedge \tau_n}\left\lvert \tilde{b}_i(s,X^0)  
	\right\rvert^2 \de s\right)\cdot \\ 
	& \qquad\qquad\qquad\cdot \mathbf{1}_{\left\lbrace\left(X^0_{t_0\wedge \tau_n},\ldots,  
		X^0_{t_m\wedge \tau_n}\right)\in \Gamma\right\rbrace} \de \mathbb{P} \\ 
	& \qquad =\mathbb{P}^{i,n}\left[\left(X^0_{t_0\wedge \tau_n},\ldots, X^0_{t_m\wedge \tau_n}\right) 
	\in \Gamma\right]\,.
	\end{align*}
	By taking the limit of $n\to\infty$, we get that the law of $X^0_{\cdot\wedge \tau_n}$ with respect to 
	$\mathbb{P}^{i,n}$ converges weakly to the law of $X^{(i)}$ with respect to $\mathbb{Q}^{(i)}$ since 
	$\mathbb{Q}^{(i)}\left(\sup_{n\geq 1}\zeta_n^{(i)}\geq T \right)=1$. Let us define the function 
	\[ 
	\phi_\varepsilon(y):=\frac{\phi(y)}{1+\varepsilon{\phi(y)}}. 
	\]
	Using Lemma \ref{lemma}, applied to the bounded function $\phi_\varepsilon \in  \mathcal{L}^1  
	(|\mu_t^{(1)} - \mu_t^{(2)} |)$, we obtain that 
	\begin{align}\label{dist}
	\left\lVert \mu^{(1)}_t-\mu^{(2)}_t\right\rVert_{\phi} 
	& =\int_{\mathbb{R}^d}\phi(y)\left\lvert\mu^{(1)}_t-\mu^{(2)}_t\right\rvert(\de y) 
	\nonumber \\ 
	& = \lim_{\varepsilon\searrow 0} \int_{\mathbb{R}^d} \phi_\varepsilon (y) 
	\left\lvert \mu^{(1)}_t-\mu^{(2)}_t\right\rvert(\de y) \nonumber \\ 
	& \leq \liminf_{\varepsilon\searrow 0}\liminf_{n\to \infty}\int_{\mathbb{R}^d} \phi_\varepsilon(y) 
	\left\lvert(\mathbb{P}^{1,n})\circ\left(X^0_{t\wedge \tau_n}\right)^{-1}-(\mathbb{P}^{2,n}) 
	\circ \left(X^0_{t\wedge \tau_n}\right)^{-1}\right\rvert(\de y).
	\end{align}
	For $A\in \mathcal{B}\left(\mathbb{R}^d\right)$, we have
	\begin{align*}
	\MoveEqLeft[3] 
	\left\lvert(\mathbb{P}^{1,n})\circ \left(X^0_{t\wedge \tau_n}\right)^{-1}-(\mathbb{P}^{2,n})\circ 
	\left(X^0_{t\wedge \tau_n}\right)^{-1}\right\rvert\left(A\right) \\ 
	& =\sup_{m\ge 1 ; \sqcup_{i=1}^m A_i=A}\sum_{i=1}^m 
	\left\lvert\mathbb{P}^{1,n}\left(X^0_{t\wedge \tau_n}\in A_i\right)-\mathbb{P}^{2,n} 
	\left( X^0_{t\wedge\tau_n}\in A_i\right)\right\rvert \\ 
	& = \sup_{m\ge 1; \sqcup_{i=1}^m A_i=A}\sum_{i=1}^m 
	\left\lvert\int_\Omega \left(M^{(1)}_{t\wedge\tau_n}-M^{(2)}_{t\wedge\tau_n}\right) 
	\mathbf{1}_{\left\lbrace X^0_{t\wedge \tau_n}\in A_i\right\rbrace}\de \mathbb{P}\right\rvert \\ 
	& \leq \sup_{m\ge 1; \sqcup_{i=1}^m A_i=A}\sum_{i=1}^m\int_\Omega \left\lvert M^{(1)}_{t\wedge\tau_n} 
	-M^{(2)}_{t\wedge\tau_n}\right\rvert\cdot\mathbf{1}_{\left\lbrace X^0_{t\wedge \tau_n} 
		\in A_i\right\rbrace}\de\mathbb{P}\\&=\int_\Omega \left\lvert M^{(1)}_{t\wedge\tau_n} 
	-M^{(2)}_{t\wedge\tau_n}\right\rvert\cdot\mathbf{1}_{\left\lbrace X^0_{t\wedge \tau_n} 
		\in A\right\rbrace}\de\mathbb{P}, 
	\end{align*}
	where $\sqcup_{i=1}^m A_i$ means the disjoint union of Borel measurable sets $A_i$, $1\leq i\leq m$. 
	Therefore
	\[ 
	\left\lVert \mu^{(1)}_t-\mu^{(2)}_t\right\rVert_\phi 
	\leq \liminf_{\varepsilon\searrow 0} \liminf_{n\to \infty} 
	\mathbb{E}_\mathbb{P}\left[\phi_\varepsilon\left( X^0_{t\wedge \tau_n}\right)\left\lvert 
	M^{(1)}_{t\wedge \tau_n}-M^{(2)}_{t\wedge \tau_n}\right\rvert\right]. 
	\]
	By using the inequality
	$\left\lvert e^x-e^y\right\rvert \leq \left\lvert x-y\right\rvert (e^x+e^y)$,
	we get
	\[
	\left\lvert M^{(1)}_{t\wedge \tau_n}-M^{(2)}_{t\wedge \tau_n}\right\rvert\leq (M^{(1)}_{t\wedge \tau_n}+M^{(2)}_{t\wedge \tau_n})N_{t\wedge \tau_n} 
	\]
	where
	\begin{align*}
	N_t &:=\Bigg\lvert\int_0^t\left[\tilde{b}_1(s,X^0)-\tilde{b}_2(s,X^0)\right]\cdot \de W_s-\frac{1}{2}\int_0^t\left[ \left\lvert\tilde{b}_1(s,X^0)\right\rvert^2-\left\lvert \tilde{b}_2(s,X^0)\right\rvert^2\right]\de s\Bigg\rvert
	\\&=\Bigg\lvert\int_0^t\left[\tilde{b}_1(s,X^0)-\tilde{b}_2(s,X^0)\right]\cdot \de \tilde{W}^{(1)}_s-\frac{1}{2}\int_0^t\left[ \left\lvert\tilde{b}_1(s,X^0)\right\rvert^2-\left\lvert \tilde{b}_2(s,X^0)\right\rvert^2\right]\de s
	\\&\qquad +\int_0^t\left[\tilde{b}_1(s,X^0)-\tilde{b}_2(s,X^0)\right]\cdot \tilde{b}_1(s,X^0)\de s
	\Bigg\rvert
	\\&=\Bigg\lvert\int_0^t\left[\tilde{b}_1(s,X^0)-\tilde{b}_2(s,X^0)\right]\cdot \de \tilde{W}^{(1)}_s+\frac{1}{2}\int_0^t\left\lvert\tilde{b}_1(s,X^0)-\tilde{b}_2(s,X^0)\right\rvert^2\de s\Bigg\rvert,
	\end{align*}
	and also similarly
	\begin{align*}
	N_t&=\Bigg\lvert\int_0^t\left[\tilde{b}_1(s,X^0)-\tilde{b}_2(s,X^0)\right]\cdot \de \tilde{W}^{(2)}_s-\frac{1}{2}\int_0^t\left\lvert\tilde{b}_1(s,X^0)-\tilde{b}_2(s,X^0)\right\rvert^2\de s\Bigg\rvert.
	\end{align*}
	So using Cauchy-Schwartz inequality, we get
	\begin{align*}
	\MoveEqLeft[0]\mathbb{E}_{\mathbb{P}}\left[\phi_\varepsilon\left(X^0_{t\wedge \tau_n}\right)\left\lvert M^{(1)}_{t\wedge \tau_n}-M^{(2)}_{t\wedge \tau_n}\right\rvert\right]\\&\leq \sum_{i=1}^2 \mathbb{E}_{\mathbb{P}}\left[ \phi_\varepsilon\left(X^0_{t\wedge \tau_n}\right)M^{(i)}_{t\wedge \tau_n}\int_0^{t\wedge \tau_n}\left\lvert\tilde{b}_1(s,X^0)-\tilde{b}_2(s,X^0)\right\rvert^2\de s\right]\\&\quad +\sum_{i=1}^2 \left(\mathbb{E}_{\mathbb{P}}\left[ \phi_\varepsilon^2\left(X^0_{t\wedge \tau_n}\right)M^{(i)}_{t\wedge \tau_n}\right]\right)^{1/2}\left(\mathbb{E}_{\mathbb{P}}\left[M^{(i)}_{t\wedge \tau_n}\left\lvert\int_0^{t\wedge \tau_n}\left[\tilde{b}_1(s,X^0)-\tilde{b}_2(s,X^0)\right]\cdot \de \tilde{W}^{(i)}_s\right\rvert^2\right]\right)^{1/2}\\&\leq \sum_{i=1}^2\mathbb{E}_{\mathbb{Q}^{(i)}}\left[\phi_\varepsilon\left(X^{(i)}_{t\wedge \zeta_n^{(i)}}\right)\int_0^{t\wedge \zeta_n^{(i)}}\left\lvert\tilde{b}_1(s,X^{(i)})-\tilde{b}_2(s,X^{(i)})\right\rvert^2\de s\right]\\&\quad+\sum_{i=1}^2\left(\mathbb{E}_{\mathbb{Q}^{(i)}}\left[\phi_\varepsilon^2\left(X^{(i)}_{t\wedge \zeta_n^{(i)}}\right)\right]\right)^{1/2}\left(\mathbb{E}_{\mathbb{Q}^{(i)}}\left[\int_0^{t\wedge \zeta_n^{(i)}}\left\lvert\tilde{b}_1(s,X^{(i)})-\tilde{b}_2(s,X^{(i)})\right\rvert^2\de s\right]\right)^{1/2}.
	\end{align*}
	Since $\phi_\varepsilon$ is bounded and continuous, we have
	\begin{align*}
	\MoveEqLeft[2]\liminf_{n\to \infty}\mathbb{E}_{\mathbb{P}}\left[\phi_\varepsilon\left(X^0_{t\wedge \tau_n}\right)\left\lvert M^{(1)}_{t\wedge \tau_n}-M^{(2)}_{t\wedge \tau_n}\right\rvert\right]\\&\leq \sum_{i=1}^2\mathbb{E}_{\mathbb{Q}^{(i)}}\left[\phi_\varepsilon\left(X^{(i)}_{t}\right)\int_0^{t}\left\lvert\tilde{b}_1(s,X^{(i)})-\tilde{b}_2(s,X^{(i)})\right\rvert^2\de s\right]\\&\quad+\sum_{i=1}^2\left(\mathbb{E}_{\mathbb{Q}^{(i)}}\left[\phi_\varepsilon^2\left(X^{(i)}_{t}\right)\right]\right)^{1/2}\left(\mathbb{E}_{\mathbb{Q}^{(i)}}\left[\int_0^{t}\left\lvert\tilde{b}_1(s,X^{(i)})-\tilde{b}_2(s,X^{(i)})\right\rvert^2\de s\right]\right)^{1/2}\\&\leq \sum_{i=1}^2\mathbb{E}_{\mathbb{Q}^{(i)}}\left[\phi\left(X^{(i)}_{t}\right)\int_0^{t}\left\lvert\tilde{b}_1(s,X^{(i)})-\tilde{b}_2(s,X^{(i)})\right\rvert^2\de s\right]\\&\quad+\sum_{i=1}^2\left(\mathbb{E}_{\mathbb{Q}^{(i)}}\left[\phi^2\left(X^{(i)}_{t}\right)\right]\right)^{1/2}\left(\mathbb{E}_{\mathbb{Q}^{(i)}}\left[\int_0^{t}\left\lvert\tilde{b}_1(s,X^{(i)})-\tilde{b}_2(s,X^{(i)})\right\rvert^2\de s\right]\right)^{1/2}
	\end{align*}
	Therefore by \eqref{dist}, we get inequality \eqref{result1}. Let us now turn to the case where $\tilde{b}_i(s,x)=b(s,x,\mu^{(i)})$. First we square both sides of \eqref{result1} with $\phi=\varphi_t$ and then we substitute inequality \eqref{ineq-b} in \eqref{result1} in the following calculation,
	\begin{align*}
	\MoveEqLeft[2]\left\lVert \mu^{(1)}_t-\mu^{(2)}_t\right\rVert^2_{\varphi_t}\\&\leq
	C\sum_{i=1}^2\left(\mathbb{E}_{\mathbb{Q}^{(i)}}\left[\varphi_t\left(X^{(i)}_{t}\right)\int_0^{t}\left\lvert \tilde{b}\left(s, X^{(i)},\mu^{(1)}\right)-\tilde{b}\left(s,X^{(i)},\mu^{(2)}\right)\right\rvert^2\de s\right]\right)^2\nonumber\\&\quad+C\sum_{i=1}^2\mathbb{E}_{\mathbb{Q}^{(i)}}\left[\varphi_t^2\left(X^{(i)}_{t}\right)\right]\cdot\mathbb{E}_{\mathbb{Q}^{(i)}}\left[\int_0^{t}\left\lvert \tilde{b}\left(s, X^{(i)},\mu^{(1)}\right)-\tilde{b}\left(s,X^{(i)},\mu^{(2)}\right)\right\rvert^2\de s\right]\\&\leq  C\sum_{i=1}^2\mathbb{E}_{\mathbb{Q}^{(i)}}\left[\varphi_t\left(X^{(i)}_{t}\right)\int_0^{t}\left\lvert \tilde{b}\left(s, X^{(i)},\mu^{(1)}\right)-\tilde{b}\left(s,X^{(i)},\mu^{(2)}\right)\right\rvert^2\de s\right]\\&\quad\cdot \mathbb{E}_{\mathbb{Q}^{(i)}}\left[\varphi_t\left(X^{(i)}_{t}\right)\int_0^{t}\psi^2(s,X^{(i)})g \left(\sup_{u\in[0,s]}\left\lVert \mu^{(1)}_u-\mu^{(2)}_u\right\rVert^2_{\varphi_u}\right)\de s\right]\nonumber\\&\quad+C\sum_{i=1}^2\mathbb{E}_{\mathbb{Q}^{(i)}}\left[\varphi_t^2\left(X^{(i)}_{t}\right)\right]
	\cdot\mathbb{E}_{\mathbb{Q}^{(i)}}\left[\int_0^{t}\psi^2(s,X^{(i)})g \left(\sup_{u\in[0,s]}\left\lVert \mu^{(1)}_u-\mu^{(2)}_u\right\rVert^2_{\varphi_u}\right)\de s\right].
	\end{align*}
	Then for the function 
	\begin{align*}
	H(t,s) 
	& := C\sum_{i=1}^2\mathbb{E}_{\mathbb{Q}^{(i)}}\left[\varphi_t\left(X^{(i)}_{t}\right)\int_0^{t} 
	\left\lvert \tilde{b}\left(u, X^{(i)},\mu^{(1)}\right)-\tilde{b}\left(u,X^{(i)},\mu^{(2)}\right)\right\rvert^2\de u\right]\cdot \\ 
	&\quad \cdot \mathbb{E}_{\mathbb{Q}^{(i)}}\left[\varphi_t\left(X^{(i)}_{t}\right)\psi^2(s,X^{(i)})\right]+C\sum_{i=1}^2\mathbb{E}_{\mathbb{Q}^{(i)}}\left[\varphi_t^2\left(X^{(i)}_{t}\right)\right]
	\cdot\mathbb{E}_{\mathbb{Q}^{(i)}}\left[\psi^2(s,X^{(i)})\right],
	\end{align*}
	we have 
	\[
	\left\lVert \mu^{(1)}_t-\mu^{(2)}_t\right\rVert^2_{\varphi_t}\leq \int_0^t H(t,s) 
	g \left(\sup_{u\in[0,s]}\left\lVert \mu^{(1)}_u-\mu^{(2)}_u\right\rVert^2_{\varphi_u}\right)\de s.
	\]
	Now define $h(s) := \sup_{u\in[s,T]} H(u,s)$. The assumption \eqref{finiteint} implies that $h$ is integrable and on the other hand, 
	\[
	\sup_{u\in[0,t]}\left\lVert \mu^{(1)}_u-\mu^{(2)}_u\right\rVert_{\varphi_u}^2 
	\leq \int_0^t h(s)g\left(\sup_{u\in[0,s]}\left\lVert \mu^{(1)}_u - \mu^{(2)}_u 
	\right\rVert_{\varphi_u}^2\right)\de s.
	\]
	Now consider the function 
	\[ 
	F(t) := \int_0^t h(s)g\left(\sup_{u\in[0,s]}\left\lVert \mu^{(1)}_u - \mu^{(2)}_u \right 
	\rVert_{\varphi_u}^2\right) \de s \, . 
	\] 
	Since $\sup_{u\in [0,t]} \left\lVert \mu^{(1)}_u-\mu^{(2)}_u\right\rVert_{\varphi_u}^2\leq F(t)$ and $g$ is increasing, we have
	\[ 
	F^\prime (t) = h(t) g\left(\sup_{u\in[0,t]}\left\lVert \mu^{(1)}_u-\mu^{(2)}_u \right\rVert_{\varphi_u}^2\right)\leq h(t)g\left(F(t)\right) 
	\]
	and therefore
	\[
	\int_0^{F(t)}\frac{1}{g(u)}\de u=\int_0^t \frac{F^\prime(s)}{g\left(F(s)\right)}\de s 
	\leq\int_0^t h(s)\de s < \infty. 
	\]
	Since $\int_{0^+}\frac{1}{g(u)}\de u=\infty$, $F(t)$ must be zero and hence 
	$\sup_{u\in[0,t]}\left\lVert \mu^{(1)}_u-\mu^{(2)}_u\right\rVert_{\varphi_u}^2\equiv 0$. Since 
	$\varphi$ is positive, this implies $ \mu^{(1)}_t=\mu^{(2)}_t$, for all $t\in [0,T]$. Therefore 
	$\mathbb{P}^{1,n}=\mathbb{P}^{2,n}$ and since $\mathbb{P}^{i,n}\circ \left(X^0_{\cdot 
		\wedge \tau_n}\right)^{-1}$ converges weakly to $\mathbb{Q}^{(i)}\circ \left( X^{(i)}\right)^{-1}$,  
	we get $\mathbb{Q}^{(1)}\circ \left( X^{(1)}\right)^{-1}=\mathbb{Q}^{(2)}\circ  
	\left( X^{(2)}\right)^{-1}$.
\end{proof} 

\begin{corollary} 
	\label{theorem2-1}
	Let
	\[
	b(t,x,\mu):=\int_{[-\tau,0]}\int_{\mathbb{R}^d}\beta (t,s,x,y)\mu_{t+s}(\de y)\kappa(\de s)
	\]
	where $\beta\equiv \sigma\tilde{\beta}$ and 
	\[\begin{aligned}\tilde{\beta} &:[0,\infty)\times [-\tau,0]
	\times C\left([-\tau,T],\mathbb{R}^d\right)\times \mathbb{R}^d\to\mathbb{R}^{d_1},\\\sigma &:[0,\infty)
	\times C\left([-\tau,T],\mathbb{R}^d\right)\to\mathbb{R}^{d\times d_1}\end{aligned}\]
	are  measurable 
	functions and $\kappa$ is a probability measure on $[-\tau,0]$. Assume that 
	\[ 
	\begin{dcases}
	\de X^0_t=\sigma(t,X^0)\de W_t,\quad t\in [0,T], \\ 
	X^0_t=\xi_t, \quad t\in[-\tau,0], 
	\end{dcases} 
	\]
	has a unique strong solution. Suppose there exist a function 
	\[V\in C^{1,2} \left([-\tau,T]\times\mathbb{R}^d, [0,\infty)\right)\]
	and measurable functions $\varphi:[-\tau,T]\to 
	C(\mathbb{R}^d,(0,\infty))$ and $\eta:[-\tau,T]\times\mathbb{R}^d\to [0,\infty)$ such that for all 
	$x\in C\left([-\tau,T],\mathbb{R}^d\right)$ and all $y\in\mathbb{R}^d$ the following properties hold:
	\begin{enumerate}[label=(C\theenumi)] 
		\item \label{C1} $\partial_tV(t,x_t)+\left\langle \nabla V(t,x_t), \beta(t,s, x,y) 
		\right\rangle +\frac{1}{2}\mathrm{tr} \left( \sigma^T(t,x)\mathrm{D}^2V(t,x_t)\sigma(t,x)\right) 
		\leq CV(t,x_t),$
		\item \label{C2} $\left\lvert \tilde{\beta}(t,s,x,y)\right\rvert\leq C\varphi(t+s,y)  
		\eta(t+s,x_{t+s}),$
		\item \label{C3} $\eta^4(t,y)+\varphi^2(t,y)\leq C V(t,y),$
		\item \label{C4} $\sup_{s\in[-\tau,0]}\mathbb{E} V(s,\xi_s)<\infty$.
	\end{enumerate}
	Then uniqueness holds for the weak solution to Vlasov-McKean equation \eqref{equ1} in the sense of Definition \ref{definition} with initial value $\xi$.
\end{corollary}

\begin{proof}
	Let $X^{(1)}_t$ and $X^{(2)}_t$ be two solutions to the Vlasov-McKean equation \eqref{equ1} with laws $\mu^{(1)}_t$ and $\mu^{(2)}_t$. We want to prove that the assumptions of Theorem \ref{prop1} hold with the function $\varphi$ and  
	\[\psi(t,x):=C\int_{-\tau}^0\eta(t+s,x_{t+s})\kappa(\de s), \quad g(u)=u.\]
	We have for $\mu,\nu\in\mathcal{M}_T$ with $\mu\vert_{[-\tau,0]}=\nu\vert_{[-\tau,0]}$,
	\begin{align*}
	\left\lvert b(t,x,\mu)-b(t,x,\nu)\right\rvert&\leq \int_{-\tau}^0 \int_{\mathbb{R}^d}C\varphi(t+s,y)\eta(t+s,x_{t+s})\left\lvert \mu_{t+s}-\nu_{t+s}\right\rvert(\de y)\kappa(\de s)\\&\leq C\int_{-\tau}^0\eta(t+s,x_{t+s})\left\lVert \mu_{t+s}-\nu_{t+s}\right\rVert_{\varphi_{t+s}}\kappa(\de s)\\&\leq \psi(t,x)\sup_{u\in[0,t]}\left\lVert \mu_u-\nu_u\right\rVert_{\varphi_u}
	\end{align*}

	All expectations and integrals in Theorem \ref{prop1} are finite via \ref{C3} provided that 
	\[\sup_{s\in[-\tau,T]}\mathbb{E} V\left(s,X^{(i)}_s\right)<\infty.\]
	We have by inequality \ref{C1},
	\begin{align*}
	e^{-Ct}V\left(t,X^{(i)}_t\right)&=V(0,\xi_0)  
	+\int_0^t e^{-Cs} \Big[-C V\left(s,X^{(i)}_s\right)+\partial_t V\left(s,X^{(i)}_s\right)  \nonumber \\ 
	& \qquad\qquad\quad + \left\langle \nabla V\left(s,X^{(i)}_s\right),\tilde{\mathbb{E}} 
	\int_{-\tau}^0\beta \left(s,u,X^{(i)},\tilde{X}^{(i)}_{s+u}\right)\kappa(\de u)\right\rangle 
	\nonumber \\ 
	& \qquad\qquad\quad + \frac{1}{2}\mathrm{tr}\left( \sigma^T\left(s,X^{(i)}\right) \mathrm{D}^2 
	V\left(s,X^{(i)}_s\right)\sigma\left(s,X^{(i)}\right)\right)\Big]\de s+M_t\\&\leq V(0,\xi_0) 
	+M_t,\nonumber 
	\end{align*}
	where, according to \eqref{integsigmab},
	\[ 
	M_t:=\int_0^t e^{-Cs} \left\langle \nabla V\left(s,X^{(i)}_s\right),\sigma\left(s,X^{(i)}\right)\de 
	W_s\right\rangle, \quad t\geq 0 
	\]
	is a local martingale, i.e. there exist stopping times $\sigma_n\uparrow \infty$ as $n\to \infty$ such 
	that $\left(M_{t\wedge \sigma_n}\right)_{t\geq 0}$ is martingale. By Fatou's lemma,
	\begin{align*}
	\mathbb{E}\left[ e^{-Ct}V\left(t,X^{(i)}_t\right)\right]&\leq \liminf_{n\to\infty}\mathbb{E} 
	\left[ e^{-C(t\wedge \sigma_n)}V\left(t\wedge \sigma_n, X^{(i)}_{t\wedge \sigma_n}\right)\right] 
	\leq \mathbb{E}V(0,\xi_0).
	\end{align*}
	This implies $\sup_{t\in[-\tau,T]}\mathbb{E} V\left(t,X^{(i)}_t\right)\leq e^{CT} 
	\sup_{s\in[-\tau,0]}\mathbb{E}V(s,\xi_s)<\infty$. Hence we have by \ref{C2} and \ref{C3} and  
	locally boundedness of $\eta$ that for $x\in C([-\tau,T],\mathbb{R}^d)$,
	\begin{align} 
	\label{boundb}
	\int_0^T\left\lvert \tilde{b}\left(t,x,\mu^{(i)}\right)\right\rvert^2\de t 
	& \leq \int_0^T\int_{[-\tau,0]}\int_{\mathbb{R}^d}\left\lvert \tilde{\beta}(t,s,x,y)\right\rvert^2\mu^{(i)}_{t+s}(\de y)\kappa(\de s)\de t\nonumber \\ 
	& \leq C^2\int_0^T\int_{[-\tau,0]}\int_{\mathbb{R}^d}\varphi^2(t+s,y)\eta^2(t+s,x_{t+s})\mu^{(i)}_{t+s}(\de y)\kappa(\de s)\de t\nonumber \\ 
	& \leq C^3\int_0^T\int_{[-\tau,0]}\mathbb{E}\left[V(t+s,X^{(i)}_{t+s})\right]\eta^2(t+s,x_{t+s})\kappa(\de s)\de t\nonumber\\&\leq C^3 \sup_{t\in[-\tau,T]}\mathbb{E}\left[V(t,X^{(i)}_{t})\right]\int_0^T\int_{[-\tau,0]} \eta^2(t+s,x_{t+s})\kappa(\de s)\de t<\infty.
	\end{align}
	So the conditions \eqref{integrability-b} and \eqref{finite} in Theorem \ref{prop1} hold. The right hand side of inequality \eqref{finiteint} has the following bound,
	\begin{align*}
	&\int_0^T\sup_{t\in[s,T]}\Bigg\lbrace\sum_{i=1}^2\mathbb{E}\left[\varphi_t\left(X^{(i)}_{t}\right)\int_0^{t}\left\lvert \tilde{b}\left(u, X^{(i)},\mu^{(1)}\right)-\tilde{b}\left(u,X^{(i)},\mu^{(2)}\right)\right\rvert^2\de u\right]\cdot\\&\quad \cdot \mathbb{E}\left[\varphi_t\left(X^{(i)}_{t}\right)\psi^2(s,X^{(i)})\right]+\sum_{i=1}^2\mathbb{E}\left[\varphi_t^2\left(X^{(i)}_{t}\right)\right]
	\cdot\mathbb{E}\left[\psi^2(s,X^{(i)})\right]\Bigg\rbrace\de s\\&\leq \int_0^T\sup_{t\in[s,T]}\Bigg\lbrace\sum_{i=1}^2\mathbb{E}\left[\varphi_t^2\left(X^{(i)}_{t}\right)\right]\cdot\left(\mathbb{E}\left[\psi^4(s,X^{(i)})\right]\right)^{1/2}\cdot\\&\qquad\qquad\quad \cdot\left(\mathbb{E}\left[\left(2\int_0^{T}\left(\left\lvert \tilde{b}\left(u, X^{(i)},\mu^{(1)}\right)\right\rvert^2+\left\lvert\tilde{b}\left(u,X^{(i)},\mu^{(2)}\right)\right\rvert^2\right)\de u\right)^2\right]\right)^{1/2}
	\\&\qquad\qquad\quad +\sum_{i=1}^2\mathbb{E}\left[\varphi_t^2\left(X^{(i)}_{t}\right)\right]
	\cdot\mathbb{E}_{\mathbb{Q}^{(i)}}\left[\psi^2(s,X^{(i)})\right]\Bigg\rbrace\de s
	\end{align*}
	Hence by \eqref{boundb} to prove inequality \eqref{finiteint}, it suffices to show that for $i=1,2$,
	\[\sup_{t\in[0,T]}\mathbb{E}\left[\varphi_t^2\left(X^{(i)}_{t}\right)+\left(\int_{[-\tau,0]}\eta^2\left(t+s, X^{(i)}_{t+s}\right)\kappa(\de s)\right)^2\right]<\infty\]
	which is obvious by $\sup_{t\in[-\tau,T]}\mathbb{E}V(t,X^{(i)}_t)<\infty$ and \ref{C3}.
\end{proof}
\begin{example}
	\label{ex2.6}
	Assume the equation
	\[\de X^0_t=\sigma(t,X^0)\de W_t, \quad X^0_t=\xi_t, t\in[-\tau,0]\]
	has unique strong solution for a locally bounded measurable function $\sigma:[0,T]\times C\left([-\tau,T],\mathbb{R}^d\right)\to \mathbb{R}^{d\times d_1}$.
	Let
	\[b(t,x,\mu):=\int_{\mathbb{R}^d}\beta (t,x,y)\mu_t(\de y)\] where $\beta\equiv\sigma \tilde{\beta}$
	for a measurable function $\tilde{\beta} :[0,T]\times C\left([-\tau,T],\mathbb{R}^d\right)\times \mathbb{R}^d\to\mathbb{R}^d$. 
	Assume that there exist $\alpha\geq 0$ and $p\in[0,2]$ such that one of the following assumptions holds for all $x\in C\left([-\tau,T],\mathbb{R}^d\right)$ and $y\in \mathbb{R}^d$,
	\begin{align*}
	\textbf{1)}&\begin{dcases}\left\lvert x_t\right\rvert^2\left(2\left\langle x_t, \beta(t,x,y)\right\rangle+\left\lvert\sigma(t,x)\right\rvert^2 \right)+(\alpha-2)\left\lvert\sigma^T (t,x)x_t\right\rvert^2\leq C(1+\left\lvert x_t\right\rvert^4),\\
	\left\lvert \tilde{\beta}(t,x,y)\right\rvert\leq C(1+\left\lvert y\right\rvert^{\alpha/2})(1+\left\lvert x_t\right\rvert^{\alpha/4}),\\
	\mathbb{E}\left[\left\lvert\xi_0\right\rvert^{\alpha}\right]<\infty;\end{dcases}
	\\
	\textbf{2)}&\begin{dcases}
	\left\lvert x_t\right\rvert^2\left(2\left\langle x_t, \beta(t,x,y)\right\rangle+\left\lvert\sigma(t,x)\right\rvert^2 \right)+(\alpha p \left\lvert x\right\rvert^p+p-2)\left\lvert\sigma^T (t,x)x_t\right\rvert^2\leq C(1+\left\lvert x_t\right\rvert^{4-p}),\\\left\lvert \tilde{\beta}(t,x,y)\right\rvert\leq C \exp(\frac{\alpha}{2}\left\lvert y\right\rvert^{p}+\frac{\alpha}{4}\left\lvert x_t\right\rvert^{p}),\\
	\mathbb{E}\left[\exp(\alpha\left\lvert \xi_0\right\rvert^p)\right]<\infty.\end{dcases}
	\end{align*}
	Then the assumptions of Corollary \ref{theorem2-1} hold with $\kappa=\delta_0$ (the Dirac measure at point zero) and
	\begin{equation*}
	\begin{dcases}
	\varphi(y) := 1+\left\lvert y\right\rvert^{\alpha/2},\\
	\eta(y) := 1+\left\lvert y\right\rvert^{\alpha/4},\\
	V\in C^2\left(\mathbb{R}^d,[0,\infty)\right) \text{ such that }  V(y)= 1+\left\lvert y\right\rvert^{\alpha} \text{ for }\left\lvert y \right\rvert\geq 1,
	\end{dcases}
	\end{equation*}
	in the case 1 and 
	\begin{equation*}
	\begin{dcases}
	\varphi(y):= \exp(\frac{\alpha}{2}\left\lvert y\right\rvert^{p}),\\
	\eta(y):= \exp(\frac{\alpha}{4}\left\lvert y\right\rvert^{p}),\\
	V\in C^2\left(\mathbb{R}^d,[0,\infty)\right) \text{ such that }  V(y)=\exp(\alpha\left\lvert y\right\rvert^{p}) \text{ for }\left\lvert y \right\rvert\geq 1,
	\end{dcases}
	\end{equation*}
	in the case 2. In particular the solution to the Vlasov-McKean equation \eqref{equ1} with initial value $\xi$ is weakly unique.
\end{example}

\section{Existence Result}

We first show an abstract theorem on the existence of weak solutions to Vlasov-McKean equations with 
measurable coefficients by approximating the respective equation with more regular coefficients. We 
then present explicit Lyapunov type assumptions on the coefficients that imply the assumptions made in 
the abstract approximation result. 

\begin{theorem}
	\label{abstractexistence}
	Let $b,\sigma:[0,\infty)\times\mathbb{R}^d\times\mathbb{R}^d\to \mathbb{R}^d, \mathbb{R}^{d\times d_1}$ 
	be measurable and locally bounded. Consider the equation
	\begin{equation} \label{equ3'}
	\de X_t=\tilde{\mathbb{E}}b(t,X_t,\tilde{X}_t)\de t+\tilde{\mathbb{E}}\sigma(t,X_t,\tilde{X}_t)\de W_t
	\end{equation}
	with initial value $X_0=\xi\in L^2\left(\Omega, \mathcal{F}_0,\mathbb{P};\mathbb{R}^d\right)$. Here 
	$\tilde{X}_t$ is an independent copy of $X_t$. Assume that there exist sequences of measurable functions \[b_n,\sigma_n:[0,\infty)\times\mathbb{R}^d\times\mathbb{R}^d\to \mathbb{R}^d, \mathbb{R}^{d\times d_1},  n\in \mathbb{N}\]
	such that for all $t\in [0,T]$, the functions $(x,y)\mapsto b_n(t,x,y), \sigma_n(t,x,y)$ are continuous and equation 
	\begin{equation}
	\de X^n_t=\tilde{\mathbb{E}}b_n(t,X^n_t,\tilde{X}^n_t)\de t+\tilde{\mathbb{E}}\sigma_n(t,X^n_t,\tilde{X}^n_t)\de W_t
	\end{equation}
	with initial value $X^n_0=\xi$ has a weak solution satisfying
	\begin{equation}\label{bddtheta_n} 
	\sup_{\substack{t\in[0,T]\\n\in\mathbb{N}}}\mathbb{E}\tilde{\mathbb{E}}\left[\left\lvert b_n(t, X^n_t, \tilde{X}^n_t)\right\rvert^q+\left\lvert\sigma_n(t, X^n_t, \tilde{X}^n_t)\right\rvert^q\right]<\infty
	\end{equation}
	for some $q>2$. Assume one of the following hypotheses hold:
	\begin{itemize}
		\item[\textbf{Case A:} ] For every $t\in[0,T]$, the mappings $(x,y)\mapsto b(t,x,y), \sigma(t,x,y)$ are continuous and for every $R>0$,
		$b_n(t,\cdot,\cdot)\to b(t,\cdot,\cdot),\  \sigma_n(t,\cdot,\cdot)\to \sigma(t,\cdot,\cdot)$ as $n\to \infty$ in $C\left(B_R\times B_R\right)$.
		\item[\textbf{Case B:} ] The function $(t,x)\mapsto \sup_{n\in\mathbb{N}}\tilde{\mathbb{E}}b_n(t,x,\tilde{X}^n_t)$ is locally bounded and for every $R>0$,
		\begin{equation}\label{non-deg}
		\liminf_{n\to\infty}\left[\inf\left\lbrace h^T \sigma_n(t,x,y)\sigma_n^T(t,x,y)h: \left\lvert h \right\rvert=1; t\in[0,T]; \left\lvert x\right\rvert,\left\lvert y\right\rvert\leq R\right\rbrace\right]>0,
		\end{equation}
		and also
		$b_n\to b$ and $\sigma_n\to \sigma$ as $n\to \infty$ in $L^{4d+2}\left([0,T]\times B_R\times B_R,\lambda\right)$.
		
	\end{itemize}
	Here $B_R$ is the ball with radius $R$ centered at the origin and $\lambda$ is the Lebesgue measure on $\mathbb{R}^{2d+1}$. Then equation \eqref{equ3'} has a weak solution on $[0,T]$.
\end{theorem}
We will use the following lemma in the proof of case B, which is a consequence of the Krylov's estimate (see Theorem 2.2.4 in \cite{krylov2008controlled}). 
\begin{lemma}\label{KrylovEstimate}
	Consider the probability space $\left(\Omega,\mathcal{F},\left(\mathcal{F}_t\right)_{t\geq 0},\mathbb{P}\right)$ and  an $\left(\mathcal{F}_t\right)_{t\geq 0}$-Wiener process $(W_t)_{t\geq 0}$ on $\mathbb{R}^{d_1}$. Let $Z(t)=\int_0^t f(t,\omega)\de t+\int_0^t g(t,\omega)\de W_t$ be an It\^o process on $\mathbb{R}^d$ where $f,g:[0,T]\times \Omega \to \mathbb{R}^d,\mathbb{R}^{d\times  d_1}$ are $\mathcal{F}_t$-adapted stochastic processes. Let us denote the exit time of $Z$ from domain $D\subset \mathbb{R}^d$ by $\tau_D$, i.e.,
	\[
	\tau_D:=\inf\left\lbrace t\geq 0: Z(t)\notin D\right\rbrace. 
	\]
	Assume that there exist constants $K$ and $\delta$ such that for all $(t,\omega)\in [0,T]\times \Omega$ with the property $t< \tau_D(\omega)$, the following inequalities hold
	\[\left\lvert f(t,\omega)\right\rvert\leq K, \qquad \inf_{\left\lvert h\right\rvert=1}h^T g(t,\omega)g^T(t,\omega) h\geq \delta \]
	Then there exists a constant $N_{\delta,K,d,D}$ depending only on $\delta, K, d$ and the diameter of the region $D$ such that for any measurable function $u:[0,T]\times D\to \mathbb{R}$,  
	\[\mathbb{E}\left[\int_0^{T\wedge\tau_D} u(t,Z(t))\de t\right]\leq N_{\delta,K,d,D} \left(\int_{[0,T]\times D} \left\lvert u(t,x)\right\rvert^{d+1} \de t\,\de x\right)^{\frac{1}{d+1}}. \] 
\end{lemma}
\begin{proof}[Proof of Theorem \ref{abstractexistence}]
	First we prove tightness of distributions of $X^n$ on $C\left( [0,T], \mathbb{R}^d\right)$. Using
	\[X^n_t-X^n_s=\int_s^t \tilde{\mathbb{E}}b_n\left(u,X^n_u,\tilde{X}^n_u\right)\de u+\int_s^t \tilde{\mathbb{E}}\sigma_n\left(u,X^n_u,\tilde{X}^n_u\right)\de W_u\]
	and Burkholder-Davis-Gundy inequality, it follows that
	\begin{align*}
	\MoveEqLeft[0]\mathbb{E}\left\lvert X^n_t-X^n_s\right\rvert^q\\&\leq 2^{q-1}\mathbb{E}\left\lvert \int_s^t \tilde{\mathbb{E}}\,b_n\left(u,X^n_u,\tilde{X}^n_u\right)\de u\right\rvert^q +2^{q-1}\mathbb{E}\left\lvert  \int_s^t \tilde{\mathbb{E}}\,\sigma_n\left(u,X^n_u,\tilde{X}^n_u\right)\de W_u\right\rvert^q\\&\leq 2^{q-1}\left\lvert t-s\right\rvert^{q-1}\mathbb{E}\int_s^t \tilde{\mathbb{E}}\left\lvert b_n\left(u,X^n_u,\tilde{X}^n_u\right)\right\rvert^q\de u+2^{q-1}C\,\mathbb{E}\left[\int_s^t  \left\lvert\tilde{\mathbb{E}}\,\sigma_n\left(u,X^n_u,\tilde{X}^n_u\right) \right\rvert^2\de u\right]^{q/2}\\& \leq 2^{q-1}\left\lvert t-s\right\rvert^{q-1}\mathbb{E}\int_s^t \tilde{\mathbb{E}}\left\lvert b_n\left(u,X^n_u,\tilde{X}^n_u\right)\right\rvert^q\de u+2^{q-1}C\left\lvert t-s\right\rvert^{\frac{q}{2}-1}\mathbb{E}\int_s^t  \tilde{\mathbb{E}}\left\lvert\sigma_n\left(u,X^n_u,\tilde{X}^n_u\right)\right\rvert^{q} \de u\\ 
	& \leq C\left\lvert t-s\right\rvert^{\frac{q}{2}} \, . 
	\end{align*}
	Since $q>2$, the laws of $X^n$ in the space of $C\left([0,T],\mathbb{R}^d\right)$ are tight and there exist some subsequence $X^{n_k}$ which converges in law to some law $\mu$ on $C\left([0,T],\mathbb{R}^d\right)$. According to Skorokhod's theorem, there exist random variables say $\left(Y^{n_k},\tilde{Y}^{n_k}\right)$ given on some probability space $(\Omega,\mathcal{F},\mathbb{P})$ with the same distribution as $\left(X^{n_k},\tilde{X}^{n_k}\right)$ converging to some random variable $\left(Y,\tilde{Y}\right)$ having distribution $\mu\otimes \mu$. Let us define
	\[M^{n_k}_t:=Y^{n_k}_t-\int_0^t \tilde{\mathbb{E}}b_{n_k}(s,Y^{n_k}_s,\tilde{Y}^{n_k}_s)\de s.\]
	$\left(M^{n_k}_t\right)_{t\geq 0}$ is a martingale with quadratic variation
	\[N^{n_k}_t:=\int_0^t \tilde{\mathbb{E}}\sigma_{n_k}(s,Y^{n_k}_s,\tilde{Y}^{n_k}_s)\tilde{\mathbb{E}}\sigma_{n_k}^T(s,Y^{n_k}_s,\tilde{Y}^{n_k}_s)\de s.\]
	We have by Burkholder-Davis-Gundy inequality and \eqref{bddtheta_n} that
	\[\sup_{k\in\mathbb{N}}\mathbb{E}\left\lvert M^{n_k}_t\right\rvert^q\leq C\sup_{k\in\mathbb{N}} \mathbb{E}\left\lvert N^{n_k}_t\right\rvert^{q/2}\leq C_T.\]
	Let us  also define
	\[M_t:=Y_t-\int_0^t \tilde{\mathbb{E}}b(s,Y_s,\tilde{Y}_s)\de t,\]
	and
	\[N_t:=\int_0^t \tilde{\mathbb{E}}\sigma(s,Y_s,\tilde{Y}_s)\tilde{\mathbb{E}}\sigma^T(s,Y_s,\tilde{Y}_s)\de s.\]
	If we can show that $M^{n_k}_t\to M_t$ and $N^{n_k}_t\to N_t$ in probability, then we have for bounded continuous function $F:C([0,s],\mathbb{R}^d)\to \mathbb{R}$ and $v,u\in\mathbb{R}^d$
	\begin{align*}
	\mathbb{E}\left[ \left\langle M_t-M_s,v\right\rangle F(Y|_{[0,s]})\right]=\lim_{k\to\infty}\mathbb{E}\left[ \left\langle M^{n_k}_t-M^{n_k}_s,v\right\rangle F(Y^{n_k}|_{[0,s]})\right]=0
	\end{align*}
	and also
	\begin{align*}
	&\mathbb{E}\left[ \left(\left\langle M_t-M_s,v\right\rangle \left\langle M_t-M_s,u\right\rangle-v^TN_t u\right) F(Y|_{[0,s]})\right]\\&=\lim_{k\to\infty}\mathbb{E}\left[ \left(\left\langle M^{n_k}_t-M^{n_k}_s,v\right\rangle \left\langle M^{n_k}_t-M^{n_k}_s,u\right\rangle-v^TN^{n_k}_t u\right) F(Y^{n_k}|_{[0,s]})\right]=0
	\end{align*}
	So $M_t$ is a martingale with quadratic variation $N_t$ and the proof is completed by using the martingale representation theorem. Now we continue the proof for each set of assumptions separately.
	\\
	\paragraph{Case A:} For $\Theta\in \left\lbrace b, \sigma\right\rbrace$, we have
	\begin{align*}
	\left\lvert \Theta_{n_k}(t,Y^{n_k}_t,\tilde{Y}^{n_k}_t)-\Theta(t,Y_t,\tilde{Y}_t)\right\rvert &\leq \left\lvert \Theta_{n_k}(t,Y^{n_k}_t,\tilde{Y}^{n_k}_t)-\Theta(t,Y^{n_{k}}_t,\tilde{Y}^{n_{k}}_t)\right\rvert\\&\quad+\left\lvert \Theta(t,Y^{n_{k}}_t,\tilde{Y}^{n_{k}}_t)-\Theta(t,Y_t,\tilde{Y}_t)\right\rvert.
	\end{align*}
	Since the sequence $(Y^{n_k}_t,\tilde{Y}^{n_k}_t)$ tends to $(Y_t,\tilde{Y}_t)$ almost surely as $k\to\infty$, it is a bounded sequence in $\mathbb{R}^{2d}$ almost surely and the right hand side of inequality above tends to zero as $k\to\infty$. So by uniform integrability, we get the convergence of $M^{n_k}_t\to M_t$ and $N^{n_k}_t\to N_t$ in $L^1$ as $k\to\infty$.
	\\
	\paragraph{Case B:} Let
	\[\tau_{n_k}(R):=\inf\left\lbrace t\geq 0: \left\lvert Y^{n_k}_t\right\rvert\vee \left\lvert \tilde{Y}^{n_k}_t\right\rvert>R\right\rbrace, \quad \tau(R):=\inf\left\lbrace t\geq 0: \left\lvert Y_t\right\rvert\vee \left\lvert \tilde{Y}_t\right\rvert>R\right\rbrace,\]
	and $\bar{\tau}(R):=\liminf_{k\to\infty}{\tau_{n_k}(R)}$.
	Since $\left(Y^{n_k},\tilde{Y}^{n_k}\right)$ tends to $(Y,\tilde{Y})$ in $C([0,T],\mathbb{R}^d)$, $\bar{\tau}(R)\leq \tau(R)$. We have
	\begin{align*}
	\mathbb{E}\left[\sup_{t\in[0,T]}\left\lvert X^{n}_t\right\rvert^2\right]&\leq C\mathbb{E}\left(\left\lvert \xi\right\rvert^2\right)+CT\mathbb{E}\int_0^T \tilde{\mathbb{E}} \left\lvert b_n\left(t,X^n_t,\tilde{X}^n_t\right)\right\rvert^2\de t\\&\quad+C\mathbb{E}\left[\int_0^T \tilde{\mathbb{E}} \left\lvert \sigma_n\left(t,X^n_t,\tilde{X}^n_t\right)\right\rvert^2\de t\right]\leq C_T.
	\end{align*}
	So the stopping times $\tau_{n_k}(R)$ satisfy 
	\begin{equation}\label{tau}
	\lim_{R\to\infty} \limsup_{k\to\infty}\mathbb{P}\otimes \tilde{\mathbb{P}}\left(\tau_{n_k}(R)<T\right)=0.
	\end{equation}
	We have
	\begin{align*}
	\MoveEqLeft[0]\mathbb{P}\otimes\tilde{\mathbb{P}}\left(\int_0^T \left\lvert \Theta_{n_k}(t,Y^{n_k}_t,\tilde{Y}^{n_k}_t)-\Theta(t,Y_t,\tilde{Y}_t)\right\rvert^2\de t>\delta\right)\\&\leq \mathbb{P}\otimes\tilde{\mathbb{P}}\left(T> \tau_{n_k}(R)\wedge \bar{\tau}(R)\right)
	\\&\quad+\mathbb{P}\otimes\tilde{\mathbb{P}}\left(T\leq\tau_{n_k}(R)\wedge \bar{\tau}(R);\int_0^T \left\lvert \Theta_{n_{k_0}}(t,Y_t,\tilde{Y}_t)-\Theta_{n_{k_0}}(t,Y^{n_{k}}_t,\tilde{Y}^{n_{k}}_t)\right\rvert^2\de t>\delta/9\right)
	\\&\quad +\mathbb{P}\otimes\tilde{\mathbb{P}}\left(T\leq\tau_{n_k}(R)\wedge \bar{\tau}(R);\int_0^T \left\lvert \Theta_{n_k}(t,Y^{n_k}_t,\tilde{Y}^{n_k}_t)-\Theta_{n_{k_0}}(t,Y^{n_{k}}_t,\tilde{Y}^{n_{k}}_t)\right\rvert^2\de t>\delta/9\right)
	\\&\quad +\mathbb{P}\otimes\tilde{\mathbb{P}}\left(T\leq\tau_{n_k}(R)\wedge \bar{\tau}(R);\int_0^T \left\lvert \Theta_{n_{k_0}}(t,Y_t,\tilde{Y}_t)-\Theta(t,Y_t,\tilde{Y}_t)\right\rvert^2 \de t>\delta/9\right)
	\\&= I_1+I_2+I_3+I_4,\text{ say.}
	\end{align*}
	Now observe that 
	\begin{align*}
	I_1&\leq\mathbb{P}\otimes\tilde{\mathbb{P}}(\tau_{n_k}(R)<T)+\mathbb{P}\otimes\tilde{\mathbb{P}}(\bar{\tau}(R)<T)\\&\leq\mathbb{P}\otimes\tilde{\mathbb{P}}(\tau_{n_k}(R)<T)+\limsup_{l\to\infty}\mathbb{P}\otimes\tilde{\mathbb{P}}(\tau_{n_l}(R)<T).
	\end{align*}
	From \eqref{tau} we obtain that $\lim_{R\to\infty} \limsup_{k\to \infty} I_1=0$.
	
	By continuity of $\Theta_{n_k}$, it is clear that for fixed $k_0$, the second term, i.e. $I_2$ tends to zero as $k\to\infty$. To take the limit of $I_3$ and $I_4$, we use Lemma \ref{KrylovEstimate}. Since  $(t,x)\mapsto \sup_{n\in\mathbb{N}}\tilde{\mathbb{E}}b_n(t,x,\tilde{Y}^n_t)$ is locally bounded, for $t\leq  T\wedge\tau_{n_k}(R)$, $\sup_{k\in\mathbb{N}}\tilde{\mathbb{E}}b_{n_k}\left(t,Y^{n_k}_t,\tilde{Y}^{n_k}_t\right)$ is bounded. Inequality \eqref{non-deg} implies that there exists $K_R\in\mathbb{N}$ such that for all $k\geq K_R$,
	\[
	\inf_{\substack{t\in [0, T\wedge\tau_{n_k}(R)]\\\left\lvert h\right\rvert \leq 1}}h^T 
	\sigma_{n_k}\left(t,Y^{n_k}_t,\tilde{Y}^{n_k}_t\right) \sigma^T_{n_k}\left(t,Y^{n_k}_t,
	\tilde{Y}^{n_k}_t\right) h\geq \varepsilon(R,T)>0. 
	\] 
	Therefore the conditions of Lemma \ref{KrylovEstimate} for It\^o process $\left(Y^{n_k}_t,
	\tilde{Y}^{n_k}_t\right)$ and the exit time $\tau_{n_k}(R)$ hold for all $k\geq K_R$ and there exists a 
	constant $C(R,T)$ such that,
	\begin{align*}
	I_3 
	& \leq \frac{9}{\delta}\mathbb{E}\tilde{\mathbb{E}}\left( \int_0^{T\wedge\tau_{n_k}(R)} 
	\left\lvert \Theta_{n_k}(t,Y^{n_k}_t,\tilde{Y}^{n_k}_t)-\Theta_{n_{k_0}}(t,Y^{n_{k}}_t,
	\tilde{Y}^{n_{k}}_t)\right\rvert^2\de t
	\right)
	\\ 
	& \leq C(R,T) \left( \int_{[0,T]\times B_R\times B_R}
	\left\lvert \Theta_{n_k}(t,x,y)-\Theta_{n_{k_0}}(t,x,y)\right\rvert^{4d+2}\de t\de x\de y
	\right)^{\frac{1}{2d+1}}\to 0,
	\end{align*}
	which tends to zero as  $k, k_0\to \infty$ since $\Theta_n\to \Theta$ in $L_{loc}^{4d+2}$ as 
	$n\to \infty$. 
	Let $w\in C \left( \mathbb{R}^d\times\mathbb{R}^d, \mathbb{R}\right)$ be compactly supported with 
	$1_{B_R\times B_R} \le w\le 1$. Then
	\begin{align*}
	I_4 
	& \leq \frac{9}{\delta}\mathbb{E}\tilde{\mathbb{E}}\int_0^{T\wedge \bar{\tau}(R)} w(Y_t,\tilde{Y}_t) \left\lvert \Theta_{n_{k_0}}(t,Y_t,\tilde{Y}_t)-\Theta(t,Y_t,\tilde{Y}_t)\right\rvert^2 \de t.
	\end{align*}
	Since continuous functions are dense in
	\[L^2\left([0,T]\times B_R\times B_R,\mu\right)\cap L^{4d+2}\left([0,T]\times B_R\times B_R,\lambda\right),\]
	where $\lambda$ is the Lebesgue measure and $\mu$ is the following finite Borel measure,
	\[\mu(A):=\mathbb{E}\tilde{\mathbb{E}}\int_0^T\mathbf{1}_{\{(t,Y_t,\tilde{Y}_t)\in A\}} w(Y_t,\tilde{Y}_t) \de t,\]
	we can find for every $\varepsilon>0$, a continuous function $g$ on $[0,T]\times\mathbb{R}^d\times\mathbb{R}^d$ such that
	\begin{align*}
	&\left(\mathbb{E}\tilde{\mathbb{E}}\int_0^{T} w(Y_t,\tilde{Y}_t) \left\lvert \Theta_{n_{k_0}}(t,Y_t,\tilde{Y}_t)-\Theta(t,Y_t,\tilde{Y}_t)-g(t,Y_t,\tilde{Y}_t)\right\rvert^2 \de t\right)^{1/2}\\&+\left(\int_0^T \int_{B_R}\int_{B_R}\left\lvert \Theta_{n_{k_0}}(t,x,y)-\Theta(t,x,y)-g(t,x,y)\right\rvert^{4d+2}\de x\de y\de t\right)^{\frac{1}{4d+2}}\leq \varepsilon.
	\end{align*}
	So
	\begin{align*}
	\frac{(\delta I_4)^{1/2}}{3}&\leq \left( \mathbb{E}\tilde{\mathbb{E}}\int_0^{T\wedge \bar{\tau}(R)} w(Y_t,\tilde{Y}_t) \left\lvert g(t,Y_t,\tilde{Y}_t) \right\rvert^2\de t\right)^{1/2}+\varepsilon \\ 
	& = \left( \mathbb{E}\tilde{\mathbb{E}}\int_0^T\mathbf{1}_{\{t< \bar{\tau}(R)\}} w(Y_t,\tilde{Y}_t) \left\lvert g(t,Y_t,\tilde{Y}_t) \right\rvert^2\de t\right)^{1/2}+\varepsilon\\ 
	&\leq \liminf_{l\to\infty}\left( \mathbb{E}\tilde{\mathbb{E}}\int_0^T\mathbf{1}_{\{t\leq \tau_{n_l}(R)\}} w(Y^{n_l}_t,\tilde{Y}^{n_l}_t) \left\lvert g(t,Y^{n_l}_t,\tilde{Y}^{n_l}_t)\right\rvert^2\de t\right)^{1/2}+\varepsilon.
	\end{align*}
	Thus, we get for large enough $l\in\mathbb{N}$,
	\begin{align*}
	\frac{(\delta I_4)^{1/2}}{3}&\leq \left(\mathbb{E}\tilde{\mathbb{E}}\int_0^{T\wedge\tau_{n_l}(R)} w(Y^{n_l}_t,\tilde{Y}^{n_l}_t) \left\lvert g(t,Y^{n_l}_t,\tilde{Y}^{n_l}_t)\right\rvert^2\de t\right)^{1/2}+2\varepsilon
	\end{align*}
	Then by Lemma \ref{KrylovEstimate}, we have
	\begin{align*}
	\frac{(\delta I_4)^{1/2}}{3}&\leq C(R,T)\left\lvert g\right\rvert_{L^{4d+2}([0,T]\times B_R\times B_R,\lambda)}+2 \varepsilon\\&\leq C(R,T)\left(\left\lvert \Theta_{n_{k_0}}-\Theta\right\rvert_{L^{4d+2}([0,T]\times B_R\times B_R,\lambda)}+\varepsilon\right)+2 \varepsilon
	\end{align*}
	So, $I_4$ also tends to zero as $k_0\to\infty$. Hence
	\[\int_0^T \left\lvert \Theta_{n_k}(t,Y^{n_k}_t,\tilde{Y}^{n_k}_t)-\Theta(t,Y_t,\tilde{Y}_t)\right\rvert^2\de t\to 0\quad \text{ in probability},\]
	as $k\to\infty$ and therefore $M^{n_k}_t\to M_t$ and $N^{n_k}_t\to N_t$ almost surely.
\end{proof}
\begin{remark}
	The proof of Theorem \ref{abstractexistence} is shorter than the proof of weak existence theorem in \cite{mishura2016existence} because in case B, we estimated $b,\sigma$ in the smaller space $L^{4d+2}$ instead of $L^{2d+1}$ and also we used the representation theorem for martingales. In fact Theorem \ref{abstractexistence} is more general than the weak existence result stated in \cite{mishura2016existence} and to prove that, it is enough to approximate $b,\sigma$ in the space $L^{4d+2}$ instead of $L^{2d+1}$. 
\end{remark}

\begin{theorem}\label{theorem4-1}
	Let $b,\sigma:[0,\infty)\times\mathbb{R}^d\times\mathbb{R}^d\to \mathbb{R}^d, \mathbb{R}^{d\times d_1}$ be measurable. Consider the equation
	\begin{equation} \label{equ3-1}
	\de X_t=\tilde{\mathbb{E}}b(t,X_t,\tilde{X}_t)\de t+\tilde{\mathbb{E}}\sigma(t,X_t,\tilde{X}_t)\de W_t
	\end{equation}
	with initial value $X_0=\xi$. Assume that there exists a convex function $V\in C^2\left( \mathbb{R}^d, [0,\infty)\right)$ such that for some $q>2$,
	\begin{enumerate}[label=(H\theenumi)]
		\item \label{H1}$\left\langle \nabla V(x), b(t, x,y)\right\rangle +\frac{1}{2}\mathrm{tr} \left( \sigma^T(t,x,y)\mathrm{D}^2V(x)\sigma(t,x,y)\right)< CV(x),$
		\item \label{H2}$\left\lvert b(t,x,y)\right\rvert^q+\left\lvert \sigma(t,x,y)\right\rvert^q <  V(x)V(y),$
		\item \label{H3}$\mathbb{E} V(\xi)+\mathbb{E}\left\lvert \xi\right\rvert^2<\infty.$
	\end{enumerate}
	Also assume that $(x,y)\mapsto b(t,x,y), \sigma(t,x,y)$ are continuous or $\sigma$, for every $R>0$, satisfies
	\[\inf_{t\in[0,T],\left\lvert x\right\rvert<R,\mu\in\mathcal{P}}\inf_{\left\lvert \lambda\right\rvert=1}\lambda^T\int_{\mathbb{R}^d}\sigma(t,x,y)\mu(\de y)\int_{\mathbb{R}^d}\sigma^T(t,x,y)\mu(\de y)\lambda>0.\]
	where $\mathcal{P}$ is the space of probability measures on $\left(\mathbb{R}^d, \mathcal{B}\left(\mathbb{R}^d\right)\right)$. Then \eqref{equ3-1} has a weak solution, say $X$, satisfying $\mathbb{E}V(X_t)\leq e^{Ct}\mathbb{E}V(\xi)$.
\end{theorem}
\begin{proof}
	We define for $n\in \mathbb{N}$ and $z=(x,y)$  the following globally Lipschitz continuous functions
	\begin{align*}
	b_{n,r}(t,z)&:=\psi^2\left(\frac{z}{n}\right)\int_{\mathbb{R}^{2d}}b(t,z+\tilde{z})r^{2d}\phi(r\tilde{z})\de \tilde{z},\\
	\sigma_{n,r}(t,z)&:=\psi\left(\frac{z}{n}\right)\int_{\mathbb{R}^{2d}}\sigma(t,z+\tilde{z})r^{2d}\phi(r\tilde{z})\de \tilde{z},
	\end{align*}
	where $0\leq \psi\leq 1$ and $0\leq \phi$ are compactly supported radial smooth functions with $\psi|_{B_1}=1$ and $\int_{\mathbb{R}^{2d}}\phi(x)\de x=1$. 
	Since $\mathrm{D}^2V$ is positive semi-definite, the function
	\[\sigma\mapsto\frac{1}{2}\mathrm{tr}\left(\sigma^TD^2V(x)\sigma\right)\]
	is convex. Since $V\in C^2(\mathbb{R}^d,[0,\infty))$, for an arbitrary $\varepsilon>0$, there exists $r_n>0$ large enough such that
	\begin{align*}
	\MoveEqLeft[1]\left\langle \nabla V(x), b_{n,r_n}(t,z)\right\rangle+\frac{1}{2}\mathrm{tr}\left(\sigma^T_{n,r_n}(t,z)D^2V(x)\sigma_{n,r_n}(t,z)\right)\\&\leq\psi^2\left(\frac{z}{n}\right)\Bigg[\left\langle \nabla V(x),\int_{\mathbb{R}^{2d}} b(t,z+\tilde{z})r_n^{2d}\phi(r_n\tilde{z})\de \tilde{z}\right\rangle\\&\quad+\frac{1}{2}\mathrm{tr}\left(\int_{\mathbb{R}^{2d}} \sigma^T(t,z+\tilde{z})r_n^{2d}\phi(r_n\tilde{z})\de \tilde{z}\, D^2V(x)\int_{\mathbb{R}^{2d}} \sigma(t,z+\tilde{z})r_n^{2d}\phi(r_n\tilde{z})\de \tilde{z}\right)\Bigg]
	\\&\leq
	\psi^2\left(\frac{z}{n}\right)\int_{\mathbb{R}^{2d}}\Big[\left\langle \nabla V(x), b(t,z+\tilde{z})\right\rangle+\frac{1}{2}\mathrm{tr}\left(\sigma^T(t,z+\tilde{z})D^2V(x)\sigma(t,z+\tilde{z})\right)
	\Big]r_n^{2d}\phi(r_n\tilde{z})\de \tilde{z}
	\\& \leq
	\psi^2\left(\frac{z}{n}\right)\int_{\mathbb{R}^{2d}}\Big[\left\langle \nabla V(x+\tilde{x}), b(t,z+\tilde{z})\right\rangle\\&\qquad\qquad\qquad\quad+\frac{1}{2}\mathrm{tr}\left(\sigma^T(t,z+\tilde{z})D^2V(x+\tilde{x})\sigma(t,z+\tilde{z})\right)
	\Big]r_n^{2d}\phi(r_n\tilde{z})\de \tilde{z}+\frac{\varepsilon}{2}\\&\leq C\psi^2\left(\frac{z}{n}\right)\int_{\mathbb{R}^{2d}}V(x+\tilde{x})r_n^{2d}\phi(r_n\tilde{z})\de \tilde{z}+\frac{\varepsilon}{2}\leq CV(x)+\varepsilon
	\end{align*}
	The same argument implies
	\[ \left\lvert b_{n,r_n}(t,x,y)\right\rvert^q+\left\lvert \sigma_{n,r_n}(t,x,y)\right\rvert^q <  V(x)V(y)+\varepsilon,\]
	Let us take $r_n >0$ large enough as above and $r_n\to\infty$ as $n\to\infty$. It is clear that $b_n=b_{n,r_n}$ and $\sigma_n=\sigma_{n,r_n}$ are bounded and globally Lipschitz. Therefore there exists a unique solution to 
	\begin{equation}
	\de X^n_t=\tilde{\mathbb{E}}b(t,X^n_t,\tilde{X}^n_t)\de t+\tilde{\mathbb{E}}\sigma(t,X^n_t,\tilde{X}^n_t)\de W_t
	\end{equation}
	with any arbitrary $\mathcal{F}_0$ measurable initial value $X^n_0=\xi$. To show that the assumptions of Theorem \ref{abstractexistence} hold, it is sufficient to prove that $\sup_{n\in\mathbb{N}}\mathbb{E} V(X^n_t)<C_T$ for all $t\in[0,T]$. We have by convexity of $V$ and \ref{H1},
	\begin{align*}
e^{-Ct}V(X^n_t)&=V(\xi)+ \int_0^t e^{-Cs}\Bigg[ \left\langle \nabla V(X^n_s),\tilde{\mathbb{E}} b_n\left(s,X^n_s,\tilde{X}^n_s\right)\right\rangle\\&\qquad+\frac{1}{2}\mathrm{tr}\left(\tilde{\mathbb{E}}\sigma^T_n\left(s,X^n_s,\tilde{X}^n_s\right)\mathrm{D}^2 V(X^n_s)\tilde{\mathbb{E}}\sigma_n\left(s,X^n_s,\tilde{X}^n_s\right)\right)-CV(X^n_s)\Bigg]\de s+M_t
	\\& \leq V(\xi)+ \tilde{\mathbb{E}}\int_0^t e^{-Cs}\Bigg[ \left\langle \nabla V(X^n_s), b_n\left(s,X^n_s,\tilde{X}^n_s\right)\right\rangle\\&\qquad+\frac{1}{2}\mathrm{tr}\left(\sigma^T_n\left(s,X^n_s,\tilde{X}^n_s\right)\mathrm{D}^2 V(X^n_s)\sigma_n\left(s,X^n_s,\tilde{X}^n_s\right)\right)-CV(X^n_s)\Bigg]\de s+M_t
	\\&\leq V(\xi)+ \int_0^t \varepsilon e^{-Cs}\de s+M_t=V(\xi)+\frac{\varepsilon}{C}\left(1-e^{-Ct}\right)+M_t
	\end{align*} 
	where $M_t$ is a local martingale starting from zero. Let $\tau_m\uparrow \infty$ be a corresponding localizing sequence. Then by Fatou's lemma,
	\begin{align*}
	e^{-Ct}\mathbb{E}V(X^n_t)&\leq \liminf_{m\to\infty}\mathbb{E}\left[ e^{-C(t\wedge\tau_m)} V(X^n_{t\wedge\tau_m})\right]= \mathbb{E}V(\xi)+\frac{\varepsilon}{C}\left(1-e^{-Ct}\right),
	\end{align*}
	and therefore,
	\[\mathbb{E}V(X^n_t)\leq e^{Ct}\mathbb{E}V(\xi)+\frac{\varepsilon}{C}\left(e^{Ct}-1\right).\]
	By Theorem \ref{theorem4-1}, there exist a weak solution to \eqref{equ3-1} like $X_t$ and  some subsequence $X^{n_k}$ which converges in law to  $X$ on $C([0,T],\mathbb{R}^d)$ as $k\to\infty$. Hence $\mathbb{E}V(X_t)\leq e^{Ct}\mathbb{E}V(\xi)$.
\end{proof}

\begin{corollary}\label{theorem4}
	Let $b,\sigma:[0,\infty)\times\mathbb{R}^d\times\mathbb{R}^d\to \mathbb{R}^d, \mathbb{R}^{d\times d_1}$ be measurable. Consider the equation
	\begin{equation} \label{equ3}
	\de X_t=\tilde{\mathbb{E}}b(t,X_t,\tilde{X}_t)\de t+\tilde{\mathbb{E}}\sigma(t,X_t,\tilde{X}_t)\de W_t
	\end{equation}
	with initial value $X_0=\xi$. Here $\tilde{X}_t$ is an independent copy of $X_t$. Suppose that one of the following assumptions holds for $q>2$:
	\begin{itemize}
		\item[1.] Assume that there exists $\alpha\geq 1$ such that
		\[\left\lvert x\right\rvert^2\left(2\left\langle x,b(t,x,y)\right\rangle+\left\lvert \sigma(t,x,y)\right\rvert^2\right)+(\alpha-2)\left\lvert\sigma^T (t,x,y)x\right\rvert^2 \leq C(1+\left\lvert x\right\rvert^4),\]
		\[\left\lvert b(t,x,y)\right\rvert^q+\left\lvert \sigma(t,x,y)\right\rvert^q\leq C(1+\left\lvert x\right\rvert^\alpha)(1+\left\lvert y\right\rvert^\alpha),\]
		and $\mathbb{E}\left( \left\lvert \xi\right\rvert^{\alpha\vee 2}\right)<\infty$.
		\item[2.] Assume that there exist $p\in[1,2]$ and $\alpha>0$ such that
		\[\left\lvert x\right\rvert^2\left(2\left\langle x,b(t,x,y)\right\rangle+\left\lvert \sigma(t,x,y)\right\rvert^2\right)+\left(\alpha p\left\lvert x\right\rvert^p+p-2\right) \left\lvert \sigma^T(t,x,y)x\right\rvert^2 \leq C(1+\left\lvert x\right\rvert^{4-p}),\]
		\[\left\lvert b(t,x,y)\right\rvert^q+\left\lvert \sigma(t,x,y)\right\rvert^q\leq C\exp(\alpha\left\lvert x\right\rvert^p+\alpha\left\lvert y\right\rvert^p),\]
		and $\mathbb{E}\left[\exp(\alpha\left\lvert \xi\right\rvert^p)\right]<\infty$.
	\end{itemize}
	Also assume that $(x,y)\mapsto b(t,x,y), \sigma(t,x,y)$ are continuous or $\sigma$ is symmetric and uniformly positive definite, i.e.,
	\[\inf_{s,x,y}\inf_{\left\lvert \lambda\right\rvert=1}\lambda^T\sigma(s,x,y)\lambda>0.\]
	Then equation \eqref{equ3} has a weak solution.
\end{corollary}
\section*{Acknowledgments}
This work is  part of the first author's PhD thesis jointly at Sharif University of Technology and Technical University of Berlin under the supervision of
Professor Bijan Z. Zangeneh, Professor Michael Scheutzow and Professor Wilhelm Stannat. She wishes to thank her supervisors for their support, encouragement and guidance.

\bibliographystyle{plain}
\bibliography{biblio}

\begin{thebibliography}{10}

\bibitem{bauer2018strong}
{Bauer, Martin and Meyer-Brandis, Thilo and Proske, Frank}.
\newblock {Strong solutions of mean-field stochastic differential equations
  with irregular drift}.
\newblock {\em {Electronic Journal of Probability}}, {23}({132}):{1--35},
  {2018}.

\bibitem{bogachev2016distances}
{Bogachev, Vladimir I and R{\"o}ckner, Michael and Shaposhnikov, Stanislav V}.
\newblock {Distances between transition probabilities of diffusions and
  applications to nonlinear Fokker--Planck--Kolmogorov equations}.
\newblock {\em {Journal of Functional Analysis}}, {271}({5}):{1262--1300},
  {2016}.

\bibitem{de2019strong}
{Chaudru de Raynal, Paul-Eric}.
\newblock {Strong well posedness of McKean--Vlasov stochastic differential
  equations with H{\"o}lder drift}.
\newblock {\em {Stochastic Processes and their Applications}}, (available
  online doi:10.1016/j.spa.2019.01.006), {2019}.

\bibitem{dos2019freidlin}
{Dos Reis, Gon{\c{c}}alo and Salkeld, William and Tugaut, Julian}.
\newblock {Freidlin--Wentzell LDP in path space for McKean--Vlasov equations
  and the functional iterated logarithm law}.
\newblock {\em {The Annals of Applied Probability}}, {29}({3}):{1487--1540},
  {2019}.

\bibitem{hammersley2018mckean}
{Hammersley, William and {\v{S}}i{\v{s}}ka, David and Szpruch, Lukasz}.
\newblock {McKean-Vlasov SDEs under measure dependent Lyapunov conditions}.
\newblock {\em {arXiv preprint:1802.03974v2}}, 2018.

\bibitem{huang2019nonlinear}
{Huang, Xing and R{\"o}ckner, Michael and Wang, Feng-Yu}.
\newblock {Nonlinear Fokker--Planck equations for probability measures on path
  space and path-distribution dependent SDEs}.
\newblock {\em {Discrete \& Continuous Dynamical Systems-A}},
  {39}({6}):{3017--3035}, 2019.

\bibitem{huang2018distribution}
{Huang, Xing and Wang, Feng-Yu}.
\newblock {Distribution dependent SDEs with singular coefficients}.
\newblock {\em {Stochastic Processes and their Applications}}, ({doi:
  10.1016/j.spa.2018.12.012}), {2018}.

\bibitem{krylov2008controlled}
{Krylov, Nikolaj Vladimirovi{\v{c}} }.
\newblock {\em {Controlled Diffusion Processes}}, volume~{14}.
\newblock {Springer Science \& Business Media}, {2008}.

\bibitem{lacker2018strong}
{Lacker, Daniel}.
\newblock {On a strong form of propagation of chaos for McKean-Vlasov
  equations}.
\newblock {\em {Electronic Communications in Probability}}, {23}({45}):{1--11},
  {2018}.

\bibitem{manita2015uniqueness}
{Manita, Oxana A and Romanov, Maxim S and Shaposhnikov, Stanislav V}.
\newblock {On uniqueness of solutions to nonlinear Fokker--Planck--Kolmogorov
  equations}.
\newblock {\em {Nonlinear Analysis}}, {128}.

\bibitem{mishura2016existence}
{Mishura, Yuliya S and Veretennikov, Alexander Yu}.
\newblock {Existence and uniqueness theorems for solutions of McKean-Vlasov
  stochastic equations}.
\newblock {\em arXiv preprint:1603.02212v8}, 2018.

\bibitem{scheutzow1987uniqueness}
{Scheutzow, Michael}.
\newblock {Uniqueness and non-uniqueness of solutions of Vlasov-McKean
  equations}.
\newblock {\em Journal of the Australian Mathematical Society}, 43(2):246--256,
  1987.

\bibitem{wang2018distribution}
{Wang, Feng-Yu}.
\newblock {Distribution dependent SDEs for Landau type equations}.
\newblock {\em {Stochastic Processes and their Applications}},
  {128}({2}):{595--621}, {2018}.

\end{thebibliography}
\end{document}